\documentclass[11pt,a4paper]{article}
\usepackage[margin=2.5cm]{geometry}

\usepackage{amsmath,amsfonts,amsthm,amssymb}
\usepackage{enumitem}

\usepackage{overpic}

\newtheorem{theorem}{Theorem}
\newtheorem{lemma}[theorem]{Lemma}
\newtheorem{corollary}[theorem]{Corollary}

\newtheorem{conjecture}[theorem]{Conjecture}
\theoremstyle{definition}
\newtheorem{definition}[theorem]{Definition}

\newtheorem{remark}[theorem]{Remark}
\newtheorem{assumption}[theorem]{Assumption}
\newtheorem*{remark*}{Remark}

\renewcommand{\le}{\leqslant}
\renewcommand{\ge}{\geqslant}

\renewcommand{\Pr}{\mathbb{P}}
\newcommand{\E}{\mathbb{E}}
\newcommand{\eps}{\varepsilon}
\newcommand{\cc}{\mathrm{c}}
\newcommand{\bb}[1]{\bigl( #1 \bigr)}
\newcommand{\floor}[1]{\lfloor #1\rfloor}

\newcommand{\Gnp}{{G_{n,p}}}
\newcommand{\Gnh}{{G_{n,1/2}}}

\newcommand{\N}{\mathbb{N}}
\newcommand{\R}{\mathbb{R}}
\newcommand{\Pb}{\mathbb{P}}

\newcommand{\muexp}{{\theta}}
\newcommand{\dfdn}{\frac{\mathrm{d}f}{\mathrm{d}n}}
\newcommand{\ts}{\tilde{s}}
\newcommand{\ttt}{\tilde{t}} 
\newcommand{\dd}{\mathrm{d}}
\newcommand{\dx}{\dd x}
\newcommand{\dy}{\dd y}
\newcommand{\dt}{\dd t}
\newcommand{\dm}{\dd m}
\newcommand{\bp}{\mathbf{p}}
\newcommand{\tP}{\widetilde{P}}
\newcommand{\tL}{\widetilde{L}}
\newcommand{\hL}{\widehat{L}}
\newcommand\dto{\overset{\mathrm{d}}{\to}}
\newcommand{\cC}{\mathcal{C}}

\DeclareMathOperator{\Var}{Var}

\usepackage[numbers]{natbib}
\usepackage[font=small,labelfont=bf]{caption}

\begin{document}

\title{How does the chromatic number of a random graph vary?}
\date{August 17, 2023}
\author{Annika Heckel
\thanks{Matematiska institutionen, Uppsala universitet, Box 480, 751 06 Uppsala, Sweden. E-mail: 
\texttt{annika.heckel@math.uu.se}. This author's research was funded by ERC Grants 676632-RanDM and 772606-PTRCSP.
}
\and Oliver Riordan
\thanks{Mathematical Institute, University of Oxford, Radcliffe Observatory Quarter, Woodstock Road, Oxford OX2 6GG, UK. E-mail: {\tt riordan@maths.ox.ac.uk}.}}

\maketitle

\begin{abstract}
The chromatic number $\chi(G)$ of a graph $G$ is a fundamental parameter, whose study was originally motivated by applications ($\chi(G)$ is the minimum number of internally compatible groups the vertices can be divided into, if the edges represent incompatibility). As with other graph parameters, it is also studied from a purely theoretical point of view, and here a key question is: what is its typical value? More precisely, how does $\chi(G_{n,1/2})$, the chromatic number of a graph chosen uniformly at random from all graphs on $n$ vertices, behave?

This quantity is a random variable, so one can ask (i) for upper and lower bounds on its typical values, and (ii) for bounds on how much it varies: what is the width (e.g., standard deviation) of its distribution? On (i) there has been considerable progress over the last 45 years; on (ii), which is our focus here, remarkably little. 
One would like both upper and lower bounds on the width of the distribution, and ideally a description of the (appropriately scaled) limiting distribution. There is a well known upper bound of Shamir and Spencer of order $\sqrt{n}$, improved slightly by Alon to $\sqrt{n}/\log n$, but no non-trivial lower bound was known until 2019, when the first author proved that the width is at least $n^{1/4-o(1)}$ for infinitely many $n$, answering a longstanding question of Bollob\'as.

In this paper we have two main aims: first, we shall prove a much stronger lower bound on the width. We shall show unconditionally that, for some values of $n$, the width is at least $n^{1/2-o(1)}$, matching the upper bounds up to the error term. Moreover, conditional on a recently announced sharper explicit estimate for the chromatic number, we improve the lower bound to order $\sqrt{n} \log \log n /\log^3 n$, within a logarithmic factor of the upper bound.

Secondly, we will describe a number of conjectures as to what the true behaviour of the variation in $\chi(G_{n,1/2})$ is, and why. The first form of this conjecture arises from recent work of Bollob\'as, Heckel, Morris, Panagiotou, Riordan and Smith. We will also give much more detailed conjectures, suggesting that the true width, for the worst case $n$, matches our lower bound up to a constant factor. These conjectures also predict a Gaussian limiting distribution.
\end{abstract}

\section{Introduction}

Given a graph $G$, a \emph{colouring} of $G$ is an assignment of colours to the vertices of $G$ so that no two adjacent vertices are coloured the same. The smallest number of colours for which this is possible is called the \emph{chromatic number} of $G$, and is denoted by $\chi(G)$. This graph parameter plays a very important role in applications, in particular in assignment problems.
Here, however, we focus on $\chi(G)$ from a theoretical point of view, simply as a natural and fundamental parameter of a graph.

As with any important graph parameter, an interesting question is: what is its typical value, if we choose $G$ uniformly at random from all graphs on $n$ (labelled) vertices? Also, how much does the chromatic number fluctuate around this critical value? Given $n \in \N$ and $p \in [0,1]$, the \emph{binomial random graph} $\Gnp$ is the graph on $n$ labelled vertices where each possible edge is included independently with probability $p$, so a uniformly random graph on $n$ vertices is simply $G_{n,1/2}$. The question just described 
was raised (in the sparse setting) by Erd\H{o}s and R\'enyi~\cite{erdos1960evolution}, in one of their seminal papers which initiated the study of random graphs. Erd\H{o}s later posed this question for the dense case, see Bollob\'as~\cite{bollobas:concentrationfixed}.
In this section we first outline the history of this problem, concentrating on the most relevant results. Then we state our new results. Finally, we present a number of conjectures as to the true behaviour of $\chi(G_{n,1/2})$, in various levels of detail. The basic conjecture is due to Bollob\'as, Morris, Panagiotou and Smith together with the present authors; the finer conjectures are new.

\subsection{Past results and questions}

In 1975, Grimmett and McDiarmid~\cite{grimmett1975colouring} found the likely order of magnitude of $\chi(G_{n,p})$ for $0<p<1$ constant. In a landmark contribution in 1987, Bollob\'as~\cite{bollobas1988chromatic} determined the asymptotic behaviour of $\chi(G_{n,p})$ in this case. In stating this result we follow a standard convention, writing $q$ for $1-p$ and $b$ for $1/q=1/(1-p)$ to make the formulae more compact.

\begin{theorem}[\cite{bollobas1988chromatic}]\label{theorem:bollobas87}
Let $0<p<1$ be constant, and let $b=1/(1-p)$. With high probability\footnote{As usual, we say that a sequence $(E_n)_{n \in \N}$ of events holds \emph{with high probability (whp)} if $\Pb(E_n) \rightarrow 1$ as $n \rightarrow \infty$.},
\[
\pushQED{\qed} 
 \chi(G_{n,p}) \sim \frac{n}{2 \log_b n}.
 \qedhere
\popQED
\]
\end{theorem}
Formally, this means that for any constant $\eps>0$, with high probability $\chi(G_{n,p})$ is between $1-\eps$ and $1+\eps$ times the bound on the right-hand side.

Theorem~\ref{theorem:bollobas87} has been sharpened several times~\cite{mcdiarmid1989method,panagiotou2009note,fountoulakis2010t}, most recently in~\cite{heckel2018chromatic}.
\begin{theorem}[\cite{heckel2018chromatic}]\label{theorem:bounds}
Fix $p \le 1-1/e^2$. Then, whp,
\begin{equation}\label{eq:estimate}
 \chi(\Gnp)=\frac{n}{2\log_b n - 2 \log_b \log_b n-2\log_b 2}+o \left( \frac{n}{\log^2 n}\right)
\end{equation}
where $b=1/(1-p)$.
\end{theorem}
For constant $p >1-1/e^2$, there is a slightly more complicated expression which also determines $\chi(\Gnp)$ whp up to accuracy $o \left( \frac{n}{\log^2 n} \right)$~\cite{heckel2018chromatic}. 

{\L}uczak~\cite{luczak1991chromatic} extended Theorem~\ref{theorem:bollobas87} to the case $p \rightarrow 0$, giving a similar expression for $\chi(G_{n,p})$ whenever $p > C/n$ for some large enough constant $C$.

All the results we have mentioned so far examine the likely \emph{value} of the chromatic number --- they give increasingly sharp upper or lower bounds for $\chi(G_{n,p})$ which hold with high probability. A separate line of enquiry asks for the \emph{concentration} of the chromatic number: even if we cannot pin down $\chi(\Gnp)$ exactly, can we say something about how much it varies?

The starting point for these questions is the classic result of Shamir and Spencer from their 1987 paper~\cite{shamir1987sharp}, in which they pioneered the use of martingale concentration inequalities in probabilistic combinatorics, something which has now become a standard tool in the area. They proved that for any function $p=p(n)$, the chromatic number of $G_{n,p}$ takes one of at most about  $\sqrt{n}$ consecutive values whp.
\begin{theorem}[\cite{shamir1987sharp}]\label{theorem:shamirspencer}
  Let $p=p(n)\in (0,1)$ and $\omega(n) \rightarrow \infty$ be arbitrary functions. Then there is a sequence of intervals $[s_n, t_n]$ of length 
  \[\ell_n := t_n - s_n \le \sqrt{n}\omega(n)\]
  such that, whp,
  \[
  \pushQED{\qed}
   \chi(G_{n,p}) \in [s_n, t_n]. \qedhere
 \popQED
 \]
\end{theorem}

 It is not hard to show that for certain extreme values of $p=p(n)$, Theorem~\ref{theorem:shamirspencer} is tight: Alon and Krivelevich~\cite{alon1997concentration} note that $\chi(G_{n,p})$ is not concentrated on fewer than $\Theta(\sqrt{n})$ values for $p=1-1/(10n)$.

For the dense case, where $p$ is constant, Alon gave a slight improvement to intervals of length about $\frac{\sqrt{n}}{\log n}$ (\cite{alonspencer}, \S 7.9, Exercise 3; see also~\cite{scott2008concentration}). If $p$ tends to $0$ quickly enough, however, Theorem~\ref{theorem:shamirspencer} can be improved considerably.

Shamir and Spencer~\cite{shamir1987sharp} showed that if $p=n^{-c}$ for $c \in (0, \frac 12)$, then $\chi(G_{n,p})$ is concentrated on at most about $n^{\frac 12 - c} \log n$ values. For $c > \frac 12$, they proved concentration on \emph{constantly} many values. {\L}uczak~\cite{luczak1991note} showed that if $c > \frac 5 6$, then $\chi(G_{n,p})$ is maximally concentrated: whp it takes one of at most two consecutive values. Finally, Alon and Krivelevich~\cite{alon1997concentration} proved two-point concentration whenever $p<n^{-c}$ with $c> \frac 12$ constant.

It should be noted that none of these concentration results gives any information about the \emph{location} of the concentration intervals. In a breakthrough contribution, Achlioptas and Naor~\cite{achlioptas2005two} found two \emph{explicit} values for $\chi(G_{n,p})$ with $p=d/n$ where $d$ is constant. Later, Coja-Oghlan, Panagiotou and Steger~\cite{coja2008chromatic} extended this result to $p<n^{-\frac34 - \eps}$, giving three explicit values in this case.

In view of strong results asserting sharp concentration of the chromatic number, starting in the late 1980s Bollob\'as raised, and he and Erd\H{o}s disseminated, the \emph{opposite} question: can we find any examples where the chromatic number of $G_{n,p}$ is \emph{not} very sharply concentrated? Of course there are cases where this is trivially true, such as when $p=1-1/(10n)$ as mentioned above. But what about interesting examples, and what about the most natural special case, $p=1/2$?

It took quite a while for this question to appear in print. In an open problems appendix to the first edition of \emph{The Probabilistic Method}~\cite{alonspencerfirstedition}, Erd\H{o}s asked: can we prove that $\chi(\Gnh)$ is not concentrated on constantly many values? Bollob\'as reiterated this question in~\cite{bollobas:concentrationfixed}, asking for \emph{any} non-trivial results asserting a lack of concentration. 
The problem is also discussed in \cite{ alon1997concentration, bollobas:randomgraphs, chung1998erdos, glebov2015concentration, mcdiarmidsurvey}.

The first result of this type was recently given by the first author in~\cite{heckel2019nonconcentration}: it turns out that, at least for some values of $n$, the chromatic number of $\Gnh$ is not concentrated on fewer than about $n^{\frac 14}$ values.
\begin{theorem}[\cite{heckel2019nonconcentration}]\label{theorem:n14}
Let $c< \frac 14$ be a constant, and let $([s_n,t_n])_{n\ge1}$ be a (deterministic) sequence of intervals such that $\Pr\bb{\chi(\Gnh)\in[s_n,t_n]}\to 1$ as $n\to\infty$. Then there are infinitely many $n$ such that $t_n-s_n>n^c$.  \qed
\end{theorem}
In other words, slightly informally, for $c<1/4$ there is no sequence of intervals of length $n^{c}$ which contain $\chi(\Gnh)$ with high probability.

\subsection{Main results}

In this paper, we improve the lower bound in Theorem~\ref{theorem:n14} to an almost optimal one, giving a lower bound on the concentration interval length which nearly matches the upper bound from Theorem~\ref{theorem:shamirspencer}.
\begin{theorem}\label{theorem:n12}
Fix $p \in (0,1)$ and $c < \frac 12 $, and 
let $([s_n,t_n])_{n\ge1}$ be a (deterministic) sequence of intervals such that $\Pr\bb{\chi(\Gnp)\in[s_n,t_n]}\to 1$ as $n\to\infty$. Then there are infinitely many $n$ such that $t_n-s_n>n^c$.
\end{theorem}
It is clear from the form of the result that Theorem~\ref{theorem:n12} also holds if we replace $c$ with $\frac 12 - o(1)$ for some function $o(1)$ which tends to $0$ sufficiently slowly. Up to this vanishing term, the exponent matches the classic upper bound of $\sqrt{n}$ for the concentration interval length given by Shamir and Spencer, and Alon's improved upper bound of $\frac{\sqrt{n}}{\log n}$.

Considering intervals centred on the expectation of $\chi(\Gnp)$, Theorem~\ref{theorem:n12} implies (but is not implied by) a corresponding bound on the variance of $\chi(\Gnp)$. Concretely, for any $c<1$, we do not have $\Var(\chi(\Gnp))=O(n^c)$.

Note that neither Theorem~\ref{theorem:n14} nor Theorem~\ref{theorem:n12} tells us anything about the concentration of the chromatic number of $G_{n,p}$ for any particular $n$, let alone every $n$. They only state that whenever $[s_n, t_n]$ is a sequence of intervals which contain $\chi(G_{n,  p})$ whp, there must be a subsequence of long intervals.
Thus, these results do not rule out the unlikely scenario that the chromatic number of $\Gnp$ is spread out over about $\sqrt{n}$ values on some sparse subsequence of the integers, and is one-point concentrated everywhere else.

We will prove a stronger result than Theorem~\ref{theorem:n12}, Theorem~\ref{theorem:nstar} below. To state this we introduce some notation, and review some classic results, concerning the independence number of $G_{n,p}$.


A set of vertices is \emph{independent} in a graph $G$ if there are no edges of $G$ between them; the \emph{independence number} of $G$, denoted by $\alpha(G)$, is the maximum size of such a set in $G$. As before, let $q=1-p$ and $b=1/q$. For $p$ constant, $\alpha(G_{n,p})$ can be determined precisely as follows: let
\begin{equation} \label{eq:a0def}
 \alpha_0 = \alpha_0(n) := 2 \log_b n - 2 \log_b \log_b n +2 \log_b \left( e/2 \right)+1;
\end{equation}
then Matula~\cite{matula1970complete,matula1972employee} and independently Bollob\'as and Erd\H os~\cite{erdoscliques} proved that 
\[\alpha(G_{n, p}) = \left\lfloor \alpha_0+o(1)\right \rfloor \text{ whp},\]
pinning down $\alpha(G_{n,  p})$ to at most two consecutive values. If we let 
\begin{equation}\label{alphadef}
 \alpha=\alpha(n):=\left \lfloor \alpha_0(n) \right \rfloor,
\end{equation}
then in fact for most $n$, whp $\alpha(G_{n, p}) = \alpha$.

Given $t \ge 1$, we call an independent set of size $t$ a \emph{$t$-set}. Let $X_t$ count the number of $t$-sets in $G_{n,p}$, and let
\begin{equation}\label{eq:mudef}
 \mu_t = \mu_t(n) := \E[X_t]=\binom{n}{t}q^{\binom{t}{2}}.
\end{equation}
If we interpret the formula above suitably for non-integer $t$, then $\alpha_0(n)$ is, to a good approximation, the value of $t$ at which $\mu_t=1$. In particular, unless $\alpha_0$ is very close to an integer, we expect many $\alpha$-sets and no $(\alpha+1)$-sets, so it is no surprise that $\alpha(\Gnp)=\alpha$ whp.

With this notation, we can now state our next, more precise, result.

\begin{theorem} \label{theorem:nstar}
Fix $p \le 1-1/e^2$ and $\eps>0$, and let $[s_n, t_n]$ be a sequence of intervals such that $\Pr\bb{\chi(\Gnp)\in [s_n,t_n]}\ge 0.9$.
Then, for each $n$ such that $\mu_{\alpha(n)}(n)< n^{1-\eps}$, there is an integer $n^* = (1+o(1)) n$ such that
\[
 t_{n^*} - s_{n^*} > C \frac{\sqrt{\mu_{\alpha(n^*)}(n^*)}}{ \log n^*}, 
\]
where
\[
 C=C(p,\eps)=\frac{\eps \log b}{9}
\]
and, as usual, $b=1/(1-p)$.
\end{theorem}

Theorem~\ref{theorem:nstar} readily implies the case $p\le 1-1/e^2$ of Theorem~\ref{theorem:n12}: we simply pick a sequence of $n$ where $\mu_\alpha$ is close to $n$, which is certainly possible; see, for example,  Lemma 4 in~\cite{heckel2019nonconcentration}. The case $p>1-1/e^2$ of Theorem~\ref{theorem:n12} will also follow easily from the proof of Theorem~\ref{theorem:nstar} (see the final part of \S\ref{ss:n12}). We have replaced the assumption that $\chi(\Gnp)$ is in a certain interval whp with a weaker concrete assumption, since this is what the proof allows. The specific constant $0.9$ is not optimized.

Theorem~\ref{theorem:nstar} still does not imply non-concentration for any particular $n$ --- this is a feature of the method --- but for every $n$ it will find some nearby $n^*$ where the concentration interval is long. In many cases we believe that the bound above is tight up to the constant factor, including the dependence on $\eps$; see Section~\ref{ss_conj}, and in particular Remark~\ref{rem:matches}.

\subsubsection*{Even stronger bounds}

Theorem~\ref{theorem:n12} implies that there are \emph{some} values of $n$ such that $\chi(G_{n, p})$ is not concentrated on fewer than $n^{\frac 12 - o(1)}$ values for some unspecified function $o(1)$. Can this be pushed any further towards Alon's upper bound of $\sqrt{n}/ \log n$? We focus on the case $p=\frac 12$.

The main bottleneck is the form of the error term in the estimate \eqref{eq:estimate}
in Theorem~\ref{theorem:bounds}, which we make essential use of in the proof of Theorem~\ref{theorem:n12}. Specifically, we use that we have an explicit estimate for $\chi(\Gnh)$, and that the derivative (w.r.t.\ $n$) of this estimate is sufficiently larger than $1 / \alpha(n)$; see Remark~\ref{rem:bottleneck} for how this affects the final bound. 

Konstantinos Panagiotou and the first author~\cite{HPbdd} recently announced a sharper explicit estimate for $\chi(\Gnh)$. To state this we need some definitions.
\begin{definition}\label{def:tbdd}
A vertex colouring of $G$ is \emph{$t$-bounded} if all colour classes have size at most $t$; the \emph{$t$-bounded chromatic number} of $G$, denoted $\chi_{t}(G)$, is the minimum number of colours in such a colouring. By an \emph{unordered ($t$-bounded) $k$-colouring} of a graph $G$, we mean a partition of $V(G)$ into $k$ non-empty independent sets (of size at most $t$). We may think of this as an equivalence class of $k$-colourings under permuting colours.
Let $E_{n,k,t}$ denote the expected number of unordered $t$-bounded $k$-colourings of $\Gnh$. Then the \emph{$t$-bounded first moment threshold} of $\Gnh$ is defined to be
\begin{equation}\label{ktdef}
 k_t(n) := \min\{k: E_{n,k,t} \ge 1 \}.
\end{equation}
Note that $E_{n,n,t}=1$, so this definition makes sense.
\end{definition}

In~\cite{HPbdd} it is shown
that if $a=a(n)$ is such that $n^{0.1}<\mu_a(n)<n^{1.9}$ (where $\mu_a(n)$ is defined in \eqref{eq:mudef}), then whp
\begin{equation} \label{eq:announcedbounds}
 \chi_{a-1} (\Gnh) = k_{ a-1}(n) + O( n^{0.99}).
\end{equation}
Unsurprisingly, when $\mu_\alpha$ is not too large, then $\chi_{\alpha-1}(\Gnh)$ and $\chi(\Gnh)$ are close, and then \eqref{eq:announcedbounds} (applied with $a=\alpha$) provides a good bound on the latter. For example, we trivially have that the expectation of the difference is at most $\mu_\alpha$, though we will need a much tighter bound (see Lemma~\ref{lem:as1}).
Assuming a much weaker form of a special case of \eqref{eq:announcedbounds}, we can prove a stronger lower bound on the non-concentration interval.
\begin{theorem}\label{theorem:polylog}
Suppose that, for any integers $n$ and $a=a(n)$ such that $\mu_a(n)=\Theta(n/\log^2 n)$, we have
\begin{equation}\label{eq:assump}
 \chi_{a-1} (\Gnh) = k_{ a-1}(n) + o(n\log\log n/\log^4 n) \text{ whp},
\end{equation}
where $k_t(n)$ is defined in \eqref{ktdef}.
Then there is a constant $c>0$ so that for any sequence of intervals $[s_n, t_n]$ such that $\Pr\bb{\chi(\Gnh) \in [s_n, t_n]}\ge 0.9$, there is a sequence of integers $n^*$ such that
\[
 t_{n^*}-s_{n^*} \ge c\frac{\sqrt{n^*} \log\log n^*}{\log^3 n^*}.
\]
\end{theorem} 
\begin{remark}\label{rem:var}
Theorem~\ref{theorem:polylog} immediately 
implies (assuming \eqref{eq:assump}) a corresponding lower bound on the variance of $Y_n=\chi(\Gnh)$: writing $w_n$ for $n^{1/2}\log\log n/\log^3 n$, if we take intervals $I_n$ of length $cw_n/2$ centred on the mean of $Y_n$, then there are infinitely many $n$ such that $\Pr(Y_n\notin I_n)> 0.1$, which implies $\Var(Y_n)> 0.1(cw_n/4)^2 = \Omega(w_n^2)$, so $\limsup \Var(Y_n)/w_n^2>0$.
\end{remark}
As we shall describe in the next section, we believe that the bound given by Theorem~\ref{theorem:polylog} is optimal up to the constant factor. 

\subsection{Conjectured behaviour}\label{ss_conj}

The behaviour of the chromatic number of $\Gnp$ is closely linked to that of the number of large independent sets, and specifically to $X_\alpha$ and $X_{\alpha-1}$ (where $X_t$ is the number of independent $t$-sets), so we take a closer look at the distributions of these random variables. 

First consider $X_\alpha$. 
Let $\muexp=\muexp(n) = \log \mu_\alpha / \log n$, so that 
    \begin{equation}\mu_\alpha = n^{\muexp}. \label{eq:mua}
    \end{equation}
 Standard calculations (see \S3.c in~\cite{mcdiarmid1989method}) give
 \begin{equation}\label{rhoalpha1}
  \muexp=\alpha_0 - \alpha + o(1) \in [-o(1), 1+o(1)].
 \end{equation}
 Thus $\muexp$ behaves as shown in Figure~\ref{fig:alphaexponent}: when $\alpha_0$ is close to an integer, $\muexp$ is close to $0$. As we increase $n$, $\muexp$ increases to near $1$ (roughly linearly in $\log n$), until $\alpha_0$ gets close to the next integer. At this point $\alpha(n)$ increases by $1$ and $\muexp$ drops back to near $0$.
\begin{figure}[tb]
\begin{center}
\begin{overpic}[width=0.9\textwidth]{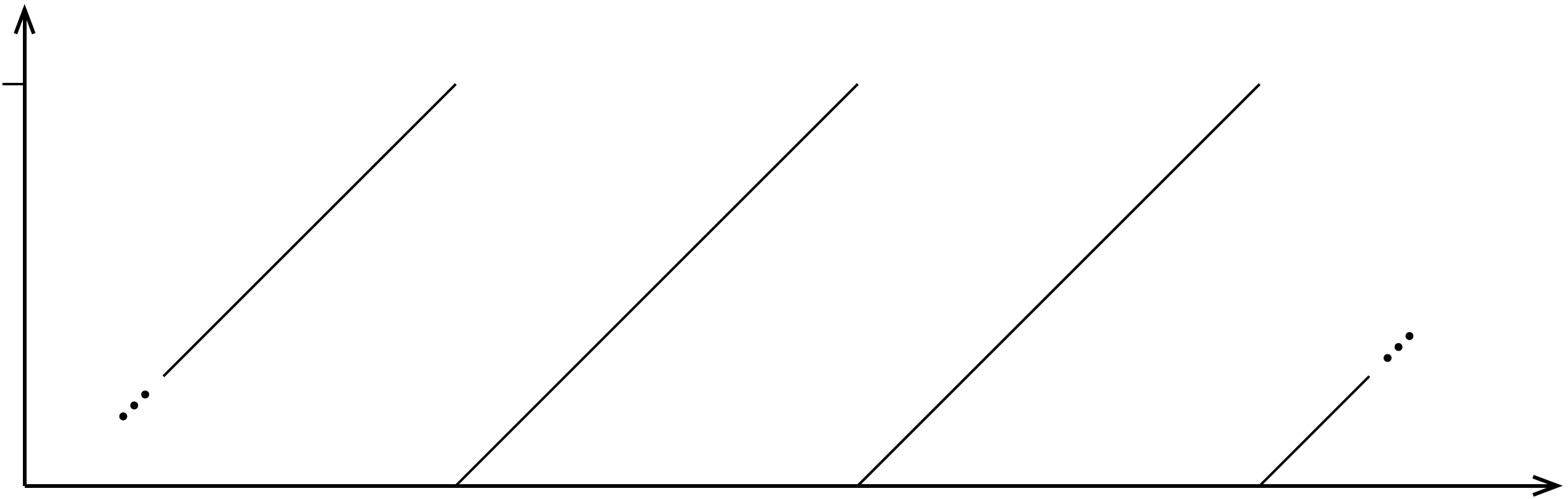}
\put(96,-2.5){$\log n$}
\put(-5,32){$\muexp(n)$}
\put(-2,25.5){$1$}
\put(-2,0){$0$}
\end{overpic}
\end{center}
\caption{\label{fig:alphaexponent} The exponent $\muexp=\muexp(n)$ so that $\mu_\alpha=n^\muexp$. When $\alpha_0(n)$ is close to an integer, $\alpha(n)=\left \lfloor \alpha_0 \right \rfloor$ increases by $1$ and $\muexp$ drops from close to $1$ to close to $0$. Note that one can think of each line segment as graphing the expected number of $t$-sets of some particular size $t$ (or rather, the log of this divided by $\log n$). These lines extend above and below the strip shown in the figure, but when we are considering the largest independent set, we jump from one size to the next as $n$ increases.}
\end{figure}
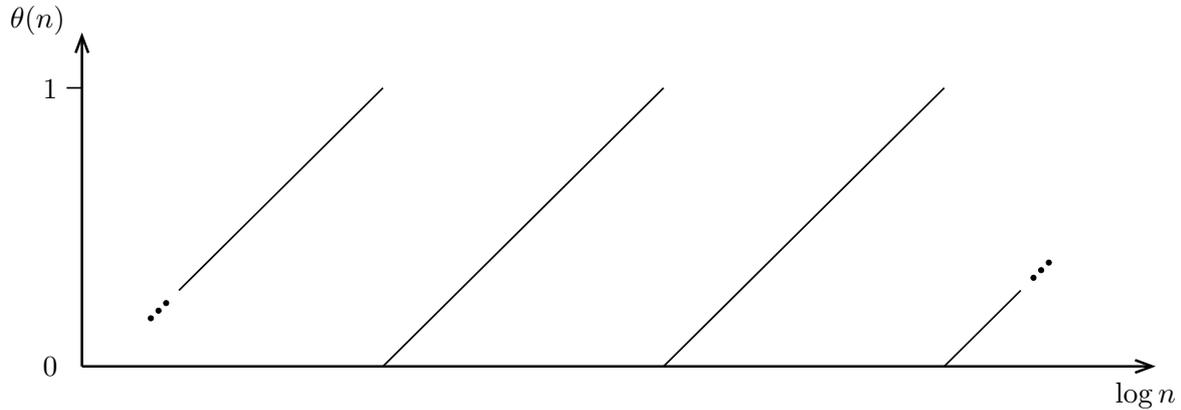

As for $X_{\alpha-1}$, note that
\begin{equation}\label{eq:mua-1}
 \mu_{\alpha-1} =\Theta\left(\frac{n}{ \log n} \mu_\alpha\right) = n^{1+\muexp+o(1)}.
\end{equation}

It turns out that both $X_\alpha$ and $X_{\alpha-1}$ are approximately Poisson for almost all $n$ (see Theorem 11.9 in~\cite{bollobas:randomgraphs}). 
In particular, $X_\alpha$ and $X_{\alpha-1}$ are not whp contained in any sequences of intervals shorter than $\sqrt{\mu_\alpha} = n^{\muexp/2}$ and $\sqrt{\mu_{\alpha-1}} = n^{(1+\muexp)/2+o(1)}$, respectively.

\subsubsection{The Zigzag Conjecture}\label{ss:zigzag}

We are now ready to state a conjecture on the correct length of the concentration interval made by Bollob\'as, Heckel, Panagiotou, Morris, Riordan and Smith. The conjecture states that the concentration interval length for $\chi(\Gnh)$ is essentially the maximum of two proposed lower bounds, one which comes from fluctuations in $X_\alpha$ and one which comes from fluctuations in $X_{\alpha-1}$, which we will describe below.

We shall consider only the case $p=\frac12$, for a number of reasons. Firstly, this is the original question; secondly, this simplifies the formulae somewhat; finally, and most importantly, for some constant $p$ --- in particular when $p>1-1/e^2$ --- the chromatic number of $\Gnp$ behaves differently to the case $p=\frac12$ (see~\cite{heckel2018chromatic}), so its concentration may well behave differently too.

The chromatic number of $\Gnh$ is closely linked to its independence number. Every colour class in a colouring is an independent set, and so for any graph $G$ on $n$ vertices, $\chi(G) \ge n / \alpha(G)$. In $\Gnh$, this simple bound for the value of the chromatic number is asymptotically correct: Bollob\'as' classic result implies that whp $\chi(\Gnh) \sim n / \alpha(\Gnh)$, and
Theorem~\ref{theorem:bounds} states that whp,
\[
 \chi(\Gnh) = \frac{n}{\alpha_0 - 1- \frac{2}{\log 2}+o(1)} \approx \frac{n}{\alpha_0-3.89}.
\]
 
It is plausible that an optimal colouring of $\Gnh$ contains all or almost all $\alpha$-sets as colour classes. To see this heuristically, fix a number $k \approx \frac{n}{2 \log_2 n}$ of colours. Each essentially different colouring of the vertex set of $\Gnh$ with $k$ colours corresponds to a \emph{profile}, i.e., a sequence of sizes for the colour classes. Among all profiles, it turns out that the expected number of colourings\footnote{As before, we actually count partitions into independent sets (with a given profile), rather than colourings.} is maximised if all or almost all $\alpha$-sets are included as colour classes. More precisely, the expectation is maximised by \emph{unrealizable} profiles containing even more $\alpha$-sets (order $n/\log n$). Although the expected number of colourings with such a profile is large, whp no such colouring exists, as there are not enough $\alpha$-sets.

We saw above that $X_\alpha$ is approximately Poisson with mean $n^\muexp$. In particular, $X_\alpha$ varies by about $\sqrt{\mu_\alpha} = n^{\muexp/2}$. If the number of available $\alpha$-sets for our colouring varies by $\sqrt{\mu_\alpha}$, intuitively the total number of colours we need should vary by at least about 
\begin{equation}\label{eq:firstlowerbound}
\frac{\sqrt{\mu_\alpha}}{\log n} = \frac{n^{\muexp/2}}{\log n}.
\end{equation}

Perhaps it is not immediately clear where the factor $\log n$ comes from. One heuristic way to see this is the following: if there are $n^{\muexp/2}$ fewer $\alpha$-sets, we can cover $n^{\muexp/2} \alpha$ fewer vertices with $\alpha$-sets and need to colour them in colour classes of size $\alpha-1$ or less. On average we colour with classes of size $\approx \alpha_0-3.89$. So each $\alpha$-set that we use covers $\Theta(1)$ extra vertices compared to a typical colour class, and hence saves $\Theta(1/\alpha_0)=\Theta(1/\log n)$ colours. This argument is an oversimplification; see \S\ref{ss_furtherconjs} for a detailed discussion.

The first part of the Zigzag Conjecture states that \eqref{eq:firstlowerbound} is indeed a lower bound for the concentration interval length of $\chi(G_{n,\frac 12})$ (see Figure~\ref{fig:conjecture}).
\begin{figure}[tb]
\begin{center}
\begin{overpic}[width=0.9\textwidth]{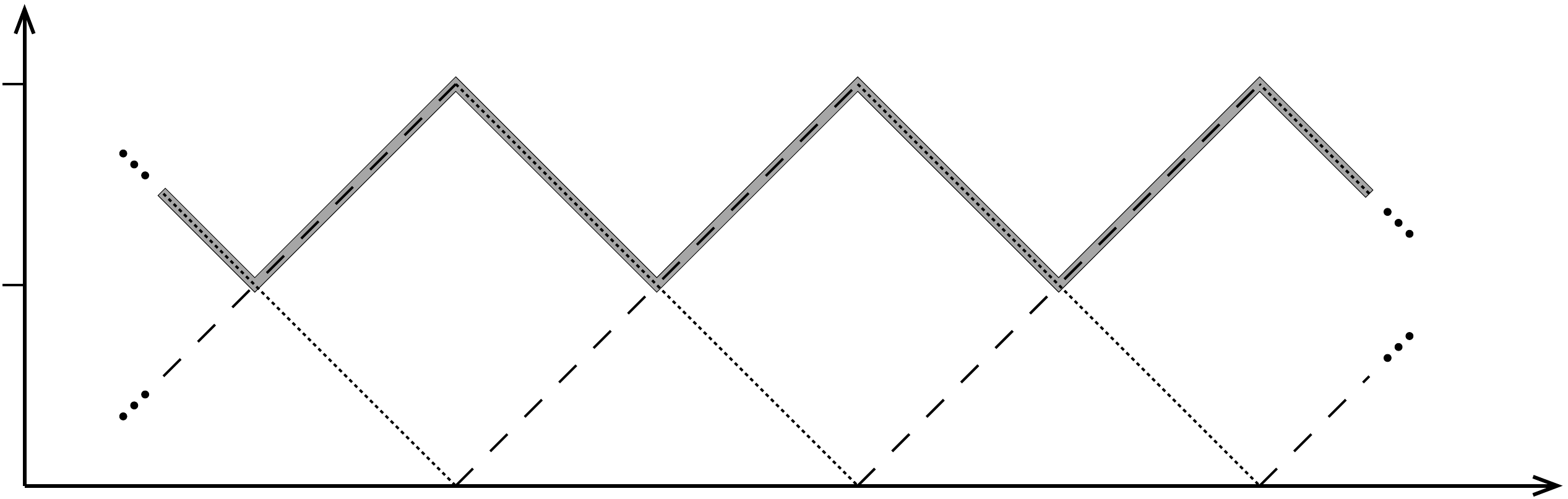}
\put(96,-2.5){$\log n$}
\put(-2,25.5){$\frac 12$}
\put(-2,12.75){$\frac 14$}
\put(-2,0){$0$}
\end{overpic}
\end{center}
\caption{\label{fig:conjecture} Exponent of the concentration interval length (in $n$). The dashed line is the conjectured lower bound $\muexp/2$. The dotted line is the conjectured lower bound $(1-\muexp)/2$. The thicker `zigzag' line is the maximum of these two lower bounds. The Zigzag Conjecture proposes that the concentration interval length of $\chi(\Gnh)$ fluctuates between $n^{1/4+o(1)}$ and $n^{1/2+o(1)}$ along this line.}
\end{figure}
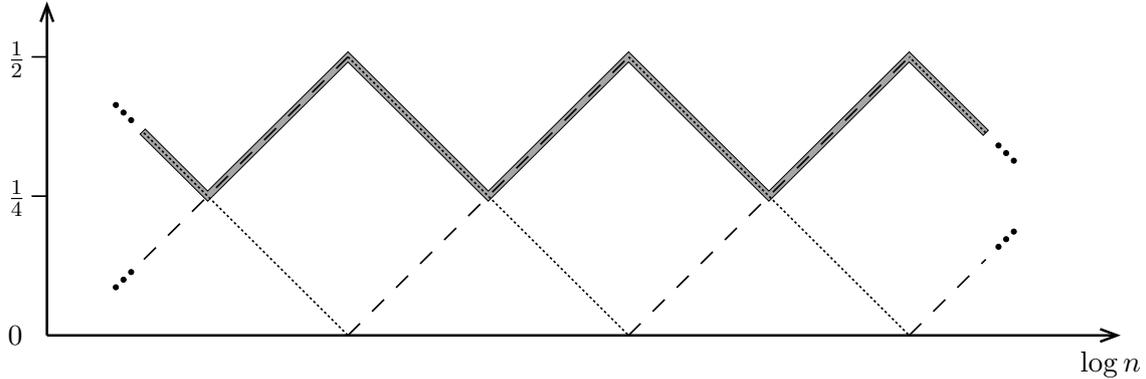

The second part is another conjectured lower bound which comes from the variations of $X_{\alpha-1}$, and is slightly trickier to understand. Again fix a number $k\sim \frac{n}{2\log_2 n}$ of colours, and consider the optimal colouring profile, that is, choose the number of colour classes of each possible size so that the expected number of such colourings is maximised. A reasonable guess is that $\chi(\Gnh)$ is close to the smallest $k$ such that, for the optimal colouring profile with $k$ colours, the expected number of colourings is at least $1$; it can be shown that the expected total number of (equivalence classes under permuting colours of) $k$-colourings is then not much more than $1$.

It turns out that the optimal profile contains $l = \Theta (n / \log n)$ colour classes of size $\alpha-1$, i.e., some constant proportion of colour classes have this size. We now make some extremely rough estimates on how much the expected number of $k$-colourings with this profile changes as $X_{\alpha-1}$ varies, at least in the highest order terms.

Since we pick $l$ colour classes from the $X_{\alpha-1}$ available $(\alpha-1)$-sets, the expected number of $k$-colourings with the optimal profile should be roughly proportional to $\binom{X_{\alpha-1}}{l} \approx X_{\alpha-1}^l$. Of course, in reality, not every choice of $l$ colour classes is possible because not all $(\alpha-1)$-sets are disjoint, but the highest order term should match, or rather, should change in the same way as $X_{\alpha-1}$ varies.

Consider $\Gnh$ conditioned on some typical values for $X_{\alpha-1}$ which are $r\approx\sqrt{\mu_{\alpha-1}}$ apart, first on $X_{\alpha-1}=m \approx \mu_{\alpha-1}$ and then on $X_{\alpha-1}=m-r$. In the second case, where we have $r$ fewer $(\alpha-1)$-sets, the expected number of $k$-colourings with optimum profile decreases by a factor of roughly
\begin{equation}
 (m-r)^l / m^ l \approx \exp \left( -rl / m\right) = \exp\left( -\Theta \left( \frac{n}{\sqrt{\mu_{\alpha-1}}\log n} \right)\right). \label{eq:decrease}
\end{equation}
So how much does the chromatic number increase when $X_{\alpha-1}=m-r$ compared to the case $X_{\alpha-1}=m$? It can be shown (see Corollary~\ref{dnk}) that adding one colour increases the expected number of colourings by a factor of size $\exp \left(\Theta\left(\log^2 n \right) \right)$. So in order to make up for the decrease in the expectation in \eqref{eq:decrease}, we need to introduce order
\[
 \frac{n}{\sqrt{\mu_{\alpha-1}} \log^3 n}
\]
additional colours. By \eqref{eq:mua-1}, note that
\begin{equation} \label{eq:secondlowerbound}
 \frac{n}{\sqrt{\mu_{\alpha-1}} \log^ 3 n} = \Theta \left( \frac{\sqrt{n}}{\sqrt{\mu_\alpha} \log^{5/2} n} \right) = \Theta \left( \frac{n^{(1-\muexp)/2}}{\log^{5/2} n} \right).
\end{equation}
The second part of the Zigzag Conjecture states that \eqref{eq:secondlowerbound} is another lower bound for the concentration interval length of $\chi(\Gnh)$ (see Figure~\ref{fig:conjecture}).

Are counts of $\alpha$-sets and $(\alpha-1)$-sets the only significant sources of non-concentration of the chromatic number? 
A recently announced result by the first author and Konstantinos Panagiotou~\cite{HPbdd} strongly suggests this (at
least for $p=\frac12$). 
Recall that the $t$-bounded chromatic number $\chi_{t}(G)$ is defined like the normal chromatic number except that we only allow colourings in which all colour classes have size at most $t$.  The announced result is that the $(\alpha-2)$-bounded chromatic number of $G_{n, m}$ with $m= \frac 12\binom{n}{2}$ is $2$-point concentrated. 
In other words, once $\alpha$-sets and $(\alpha-1)$-sets are banned as colour classes, and the number of edges is fixed, the required number of colours is extremely narrowly concentrated. It is easy to see that, in $\Gnh$, the variation in the number of edges only has a very small effect on the chromatic number, accounting for fluctuations of order at most $\log n$; see \S3 of~\cite{heckel2019nonconcentration} for a simple coupling argument showing this.

The full conjecture, therefore, states that the maximum of the lower bounds \eqref{eq:firstlowerbound} and \eqref{eq:secondlowerbound} is indeed the correct concentration interval length for $\chi(\Gnh)$ --- at least whenever $\muexp(n)$ is bounded away from $0$ and $1$.

Ignoring terms of size $n^{o(1)}$, a simplified statement is the following.
\begin{conjecture}[Zigzag Conjecture; Bollob\'as, Heckel, Morris, Panagiotou, Riordan and Smith]\label{conj:zz}
 Set $p=\frac 12$ and define $\muexp=\muexp(n)$ as in \eqref{eq:mua}. Let 
\begin{equation}\label{eq:ln}
 \lambda = \lambda(n) := \max\left(\frac{\muexp}{2}, \frac{1-\muexp}{2}\right).
\end{equation}
Then there is a sequence of intervals of length
$ n^{\lambda+o(1)} $
which contains $\chi(G_{n,\frac 12})$ whp. However, for any fixed $\eps>0$ and any sequence $(I_n)_{n \in \N}$ of intervals of length 
$ n^{\lambda-\eps}$,
we have
\[
 \Pb\left(\chi(\Gnh) \in I_n\right) =o(1).  
\]
\end{conjecture}
An analogous statement presumably holds for any constant $p\in (0,1-1/e^2]$, or perhaps $p\in (0,1-1/e^2)$.

Conjecture~\ref{conj:zz} would imply that the concentration interval length of $\chi(\Gnh)$ fluctuates between $n^{\frac 14+o(1)}$ and $n^{\frac 12+o(1)}$ as shown in Figure~\ref{fig:conjecture}.

Theorem~\ref{theorem:nstar} \emph{almost} proves the first lower bound \eqref{eq:firstlowerbound} coming from fluctuations in $X_\alpha$: we show that, for any integer $n$ with $\muexp(n)$ bounded away from $1$, there is another integer $n^*$ nearby such that \eqref{eq:firstlowerbound} holds. It is of course extremely unlikely that the width of the distribution of $\chi(\Gnh)$ is significantly different between $n$ and $n^*$, so our result presumably holds for all $n$, but we cannot prove this.

\subsubsection{Further conjectures}\label{ss_furtherconjs}

In this section we state a number of further conjectures refining Conjecture~\ref{conj:zz}. We will explain the intuition behind these conjectures in the appendix, \S\ref{s:intuition}. 
So far we have focussed on the width of the distribution as measured by concentration in an interval; here it will often be more convenient to work with the variance. Of course we expect these to be equivalent: if $Y_n:=\chi(\Gnh)$ then we \emph{expect} that $Y_n$ is concentrated on some sequence of intervals of length $\ell_n$ if and only if $\ell_n/\sigma_n\to\infty$, where $\sigma_n^2=\Var(Y_n)$. However, we do not know this, only the one-way implication that small variance implies tight concentration. 

We start with a conjecture on the worst case concentration width: we believe that, up to a constant factor, the lower bound given in Theorem \ref{theorem:polylog} is optimal.

\begin{conjecture}\label{conj:wc}
Let $p\in (0,1-1/e^2]$ be constant, let $Y_n=\chi(\Gnp)$, let $\sigma_n^2=\Var(Y_n)$, and set
\begin{equation}\label{wndef}
 w_n := \frac{\sqrt{n}\log\log n}{\log^3 n}.
\end{equation}
Then 
\begin{equation}\label{eq:cwc}
 0 < \limsup \frac{\sigma_n}{w_n} < \infty.
\end{equation}
Moreover, for any constant $c>0$ there is a constant $d>0$ such that along any sequence
of integers $n$ with $\mu_{\alpha(n)}(n)\sim c n/\log^2n$ we have $\sigma_n\sim d w_n$.
\end{conjecture}

As noted in Remark~\ref{rem:var}, Theorem~\ref{theorem:polylog} implies the first inequality in \eqref{eq:cwc}, subject to \eqref{eq:assump}.

We have a corresponding conjecture for the best case, although we are less confident of this,
so we state only the basic form.
\begin{conjecture}\label{conj:min}
Let $p\in (0,1-1/e^2]$ be constant, let $Y_n=\chi(\Gnp)$, let $\sigma_n^2=\Var(Y_n)$, and set
\begin{equation}\label{twdef}
 \widetilde{w}_n := \frac{n^{1/4}}{\log^{7/4} n}.
\end{equation}
Then 
\[
 0 < \liminf \frac{\sigma_n}{\widetilde{w}_n} < \infty.
\]
\end{conjecture}

In fact, we believe that for most (probably all) $n$, the chromatic number is asymptotically normally distributed, with a variance that follows (a refined version of) the graph suggested by the Zigzag Conjecture. We are least confident about points close to the minima in this graph, which we call `bad'.

Fix a constant $\delta>0$, and call $n$ `bad' if $n^{1/2-\delta}\le \mu_{\alpha(n)}(n)\le n^{1/2+\delta}$, and `good' otherwise.

\begin{conjecture}\label{conj:normal}
Let $p\in (0,1-1/e^2]$ be constant, and let $Y_n=\chi(\Gnp)$.
There are functions $f(n)$ and $g(n)$ such that, at least for `good' $n$,
\[
 \frac{Y_n-f(n)}{g(n)} \dto N(0,1),
\]
where $N(0,1)$ is a standard Gaussian distribution. Moreover, $g(n)=n^{\lambda(n)+o(1)}$,
where $\lambda(n)$ is defined in \eqref{eq:ln}.
\end{conjecture}

For good $n$, the dominant source of the variation should be (as described earlier) the variation in the number of independent sets of a certain size $a=\alpha-1$ or $a=\alpha$, depending on the parameters. Specifically, let
\begin{equation}\label{andef}
 a(n) := \floor{\alpha_0(n)-1/2},
\end{equation}
so, for good $n$, we have
\[
 n^{1/2+\delta+o(1)} \le \mu_{a(n)}(n) \le n^{3/2-\delta+o(1)}.
\]
One can alternatively take this to \emph{define} $a(n)$. For bad $n$, at least at a certain transition point that we don't identify precisely, it is not clear how to define $a(n)$. Indeed, in a certain range two sizes should contribute. However, the distribution should still be asymptotically normal, since a linear combination of two Gaussians is Gaussian.

\begin{conjecture}\label{conj:whynormal}
Let $p\in (0,1-1/e^2]$ be constant, let $Y_n=\chi(\Gnp)$, and let $Z_n$ be the number of independent sets of size $a(n)$ in $\Gnp$, where $a(n)$ is defined in \eqref{andef}.
Then there are functions $f(n)$, $g(n)$ such that, for good $n$,
\[
 \left(  \frac{Y_n-f(n)}{g(n)}, \frac{Z_n-\E Z_n}{\sqrt{\Var(Z_n)}} \right) \dto (Z,Z),
\]
where $Z\sim N(0,1)$.
\end{conjecture}
In other words, knowing $f(n)$ and $g(n)$ (which we do not), the value of $Z_n$ is enough to predict $Y_n=\chi(\Gnp)$ up to an error that is $o(g(n))$, i.e., smaller order than the standard deviation. In fact, for good $n$, this $o(\cdot)$ term should be $n^{-\Omega(1)}$. We would expect the conclusion of Conjecture~\ref{conj:whynormal} to hold outside a much smaller `bad' set, perhaps only having to exclude $n$ such that $\mu_{\alpha(n)}(n)$ is $\Theta(h(n))$ for some $h(n)$ close to $n^{1/2}$.

Finally, we believe that, except for `bad' $n$, we can describe the width of the distribution up to a constant factor, and in a significant fraction of cases, up to a $1+o(1)$ factor. Defining $a=a(n)$ as above, define $x=x(n)$ by
\begin{equation}\label{xndef}
 \mu_a(n) = \frac{2xn}{a^2} = \Theta\left(\frac{xn}{\log^2 n}\right).
\end{equation}
The precise normalisation here is not so important; the second formula is the key one.
\begin{conjecture}\label{conj:4cases}
Define $a(n)$ and $x(n)$ as in \eqref{andef} and \eqref{xndef}.
For good $n$, the function $g(n)$ in Conjecture~\ref{conj:normal}, or equivalently $\sigma_n=\sqrt{\Var(Y_n)}$, satisfies the following bounds, with $c_0=2/\log 2$.

\noindent
(i) if $x\to 0$, then
\[
 g(n) \sim \sqrt{\mu_a}\frac{\log\log n+\log(1/x)}{c_0 \log^2 n},
\]
(ii) if $x=\Theta(1)$, then, defining $w_n$ as in \eqref{wndef},
\[
 g(n)=\Theta\left(\sqrt{\mu_a}\frac{\log\log n}{\log^2 n}\right) = \Theta(w_n),
\]
(iii) if $x\to\infty$ with $x=n^{o(1)}$, then
\[
 g(n) \sim  \frac{\sqrt{\mu_a}}{x}\cdot \frac{\log\log n+\log x}{c_0\log^2 n} 
\sim c_1 n^{1/2} \frac{\log\log n+\log x}{\sqrt{x} \log^3 n},
\]
and (iv) if $x\ge (\log n)^C$ for some constant $C>0$, then
\[
 g(n) = \Theta\left(\frac{\sqrt{\mu_a}\log x}{x\log^2 n}\right) = \Theta\left(\frac{\sqrt{n}\log x}{\sqrt{x}\log^3 n}\right).
\]
\end{conjecture}
Note that the four ranges above cover all good $n$, with some overlap between (iii) and (iv). (The formula for (iii) applies in case (iv) too, but simplifies to (iv) in that case.) We can give a single formula applicable in all cases, but it is not clear that this is informative -- the transition from case (i) to cases (iii)/(iv) is rather arbitrary, since in case (ii) we do not even have a guess as to what the implicit constant should be (as a function of $x$). Still,
defining
\[
 g_0(n) = \sqrt{\mu_a} \frac{\log\log n+|\log x|}{c_0  (1+x)\log^2 n},
\]
in all cases we conjecture that $g(n)=\Theta(g_0(n))$, with $\sim$ in cases (i) and (iii). 

\begin{remark}
If $\log \mu_a / \log n$ is bounded away from $1/2$, $1$ and $3/2$, then the formulae in (i) and (iv) match our earlier heuristics \eqref{eq:firstlowerbound} and \eqref{eq:secondlowerbound} up to constant factors. Thus, cases (i) and (iv) of Conjecture~\ref{conj:4cases} refine the `zig' and `zag'
parts of the Zigzag Conjecture. Case (ii), and also case (iii), interpolate between these parts,
describing the conjectured shape of the top of the zigzag curve. For the bottom, we haven't stated a very detailed conjecture, but extrapolating the formulae in (i) and (iv) suggests that
when $\mu_{\alpha(n)}(n)=\Theta(\sqrt{n}/\log^{3/2}n)$ and $\mu_{\alpha(n)-1}(n)=\Theta(n^{3/2}/\log^{5/2}n)$ (which is within, and indeed in some sense the centre of, the `bad $n$' case), then the
contributions from $(\alpha-1)$-sets and $\alpha$-sets to $g(n)$ should both be of order
$n^{1/4}/\log^{7/4}n$, and this is how Conjecture~\ref{conj:min} arises.
\end{remark}

\begin{remark}\label{rem:matches}
The agreement between the lower bound in Theorem~\ref{theorem:nstar} and case (i) of Conjecture~\ref{conj:4cases} is in some sense surprisingly strong. The formula for $t_{n^*}-s_{n^*}$ in the former matches $g(n^*)$ up to a constant factor, noting that $1/x$ is at least approximately $n^\eps$. Since we may let $\eps$ tend to zero at some rate, and the dependence on $\eps$ matches, this shows that $\sqrt{\Var(Y_n)}=\Omega(g_0(n))$ not for every $n$, but at least for some $n^*$ near any good $n$ with $\mu_a(n)\le n^{1-\gamma(n)}$, where $\gamma(n)$ is a function tending to zero at a rate that we have not determined.
Similarly, it was quite a surprise to us (and not the case when we first formulated the conjectures) that we can prove a (conditional) lower bound (Theorem~\ref{theorem:polylog}) that (for a subsequence) matches the upper bound in Conjecture~\ref{conj:wc}.
\end{remark}

A completely satisfactory understanding of the asymptotic distribution of $\chi(\Gnh)$ would involve two further ingredients: we would like to know $g(n)$ (or $\sigma_n$) asymptotically, not just up to constant factors. It's quite possible that one could read out such a formula from our intuitive justification of the conjectures above (see \S\ref{s:intuition}), though of course we are nowhere near a proof. The second is that we would of course like to know $f(n)$ up to an additive error of $o(g(n))$. This seems to be a much harder problem, for which we do not even have a conjecture. See the discussion in \S\ref{s:intuition}.

\medskip

The rest of the paper is organized as follows. First, in \S\ref{ss:outline}, we outline the general strategy of the proofs. In \S\ref{ss:framework} we state and prove a concrete `framework lemma' that formalizes this strategy, essentially giving a conditional result subject to two ingredients. In \S\ref{ss:coupling} we provide the first ingredient, a simple coupling lemma. The details of the other ingredient vary from case to case; after some preliminaries in \S\ref{ss:prelim} we provide these, and so prove Theorems~\ref{theorem:nstar} and~\ref{theorem:n12}, in \S\ref{ss:nstar} and \S\ref{ss:n12}, respectively. The (very much more involved) argument for Theorem~\ref{theorem:polylog} is given in \S\ref{s:polylog}, with the proof of the key lemmas in \S\ref{ss:polylog2} and \S\ref{ss:alphashift}. Finally, we discuss the intuition behind our more precise conjectures in \S\ref{s:intuition}.

\section{Proofs}

\subsection{Proof outline}\label{ss:outline}

\begin{figure}[tb]
\begin{center}
\begin{overpic}[width=0.9\textwidth]{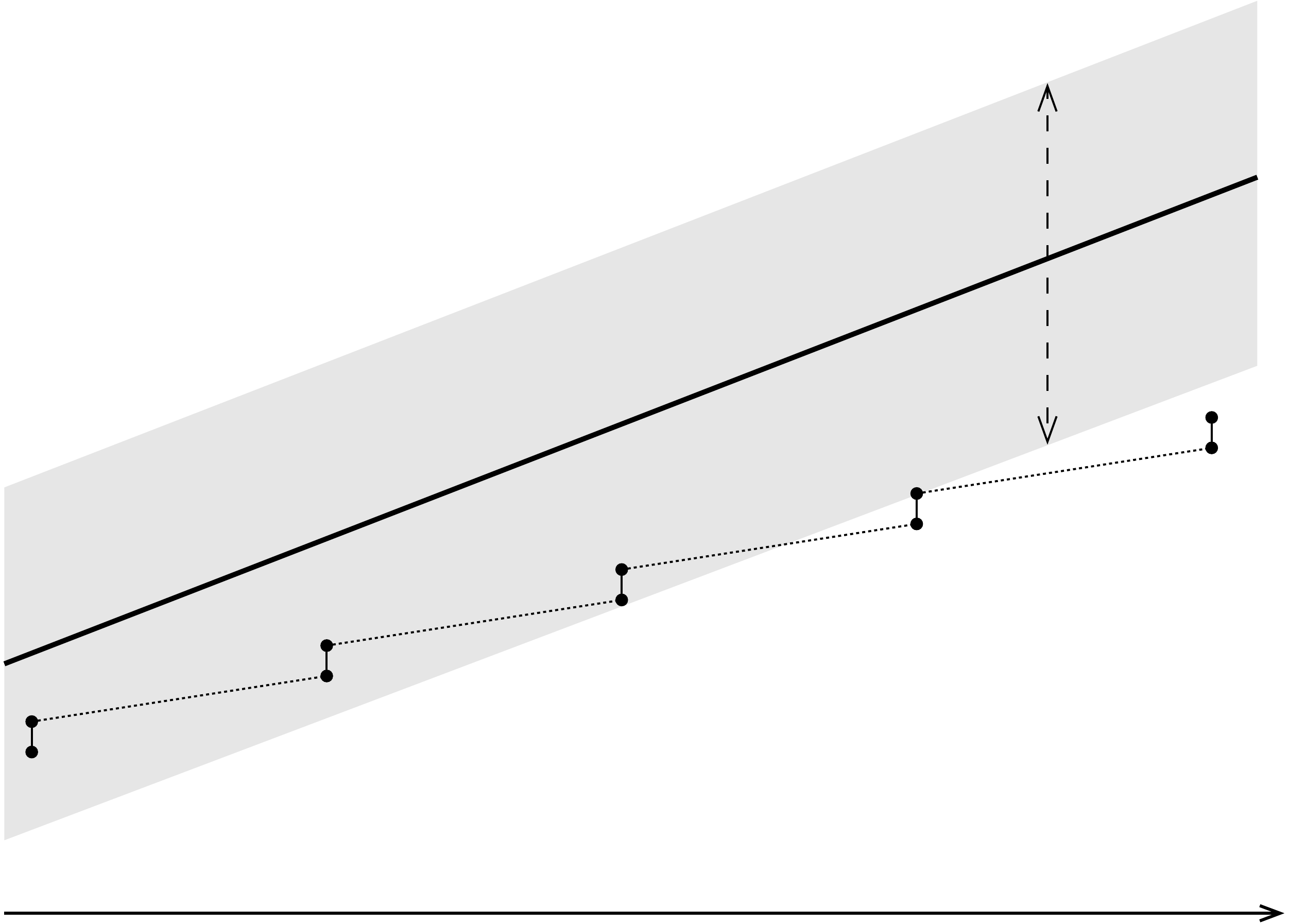}
 \put(97,-2.5){$n$}
 \put(43,40.5){$f(n)$}
 \put(66,56.5){$f(n)\pm\Delta(n)$}
 \put(22,17){$s_n$}
 \put(22,22.5){$t_n$}
 \put(44.2,23){$s_{n'}$}
 \put(44.5,28.5){$t_{n'}$}
 \put(35.8,17){$\alpha r$}
 \put(52,22.4){$\le r$}
 \put(26.5,17.5){\makebox(21.5,5){\upbracefill}}
 \put(49.5,22.4){{\Large $\rbrace$}}
\end{overpic}
\end{center}
\caption{\label{fig:intervals} Illustration of the basic strategy. We know $\chi(\Gnp)$ is concentrated in the (wide) grey band around a function $f(n)$ with slope more than $1/\alpha$. A coupling argument shows that for suitable $r$ it is likely that $\chi(G_{n',p})\le \chi(G_{n,p})+r$, where $n'=n+\alpha r$ (dotted lines with slope $1/\alpha$). If the concentration intervals $[s_n,t_n]$ are too short, a contradiction results.}
\end{figure}
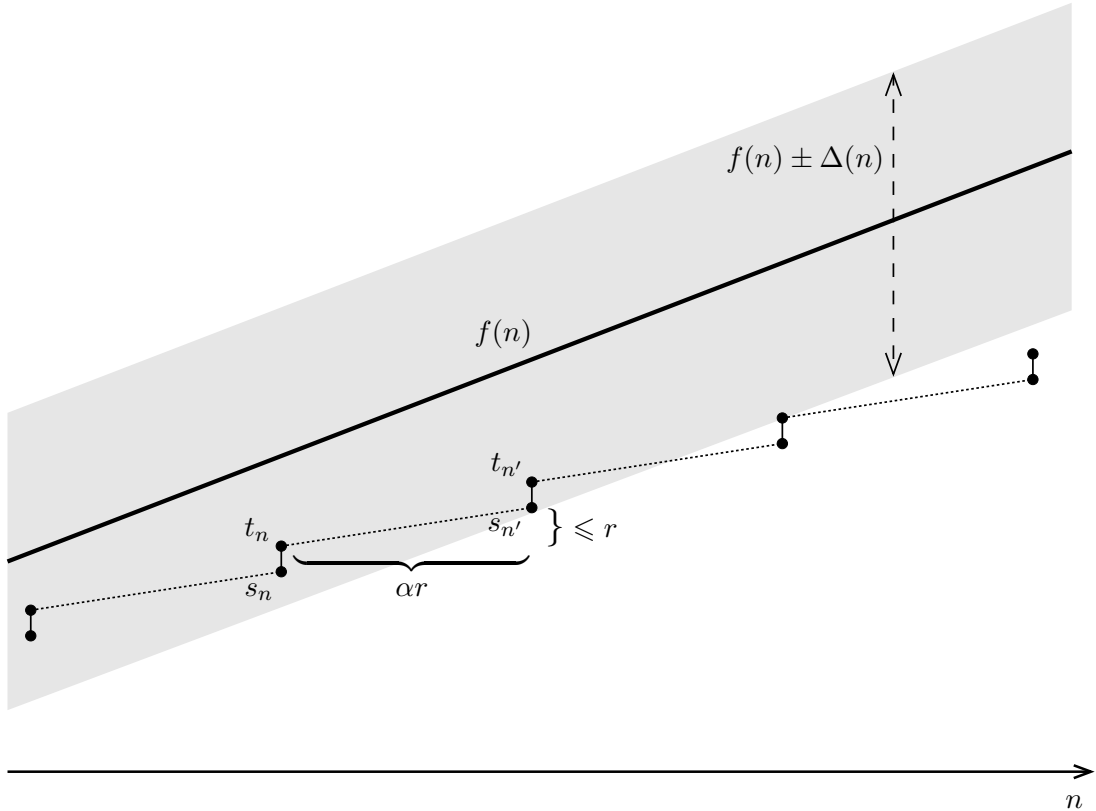

Before turning to the details, we outline the method, which is in principle simple but involves significant calculation. Throughout we fix $0<p<1$. There will be two key ingredients. First, we take as an input a suitable result establishing whp \emph{concentration} of $\chi(\Gnp)$ on some explicit interval $f(n)\pm \Delta(n)$. Here the interval length $2\Delta(n)$ will be much larger than the scale on which we are aiming to establish non-concentration. It will be essential that, interpolating $f(n)$ to non-integer values, over the range of $n$ that we consider we have $\dfdn>1/\alpha$,
where $\alpha=\alpha(n)=\floor{\alpha_0(n)}$ as before, and we consider a range of values for $n$ such that $\alpha$ does not change in this `window'. More specifically, we will suppose that
\[
 \dfdn \ge \frac{1}{\alpha} + \delta
\]
for some $\delta>0$.

The second key ingredient is a simple coupling result, Lemma~\ref{lemma:coupling} below, and in particular its consequence, Corollary~\ref{corollary:coupling}, which states, slightly informally, that for $r$ not too large we may couple the random graphs $G_n=G_{n,p}$ and $G_{n'}=G_{n+\alpha r,p}$ so that with significant probability (say $> 0.4$, though we could write $> 0.99$ by changing the constants) we have $\chi(G_{n'})\le \chi(G_{n})+r$. Here we can take $r$ up to roughly $\sqrt{\mu_\alpha(n)}$, the standard deviation of the number $X_\alpha$ of $\alpha$-sets. The intuition behind this is, roughly speaking, that because the number of $\alpha$-sets varies by at least $r$, planting $r$ extra ones does not affect the distribution of our graph too much. Planting these sets in $G_{n'}$, $n'=n+\alpha r$, we can view the graph on the remaining vertices as $G_n$, giving the coupling.

Suppose for the moment that $\chi(G_n)=\chi(G_{n,p})$ were in fact deterministic, equal to some function $f_0(n)$. Then the coupling just described would show that $f_0(n')\le f_0(n)+r$, i.e., (essentially) that the function $f_0(n)$ has slope $f_0'(n)$ at most $1/\alpha$. This would lead to a contradiction, considering a suitably large range of values of $n$. Indeed, by our first (concentration) assumption, $|f_0(n)-f(n)|\le \Delta(n)$. But (ignoring the variation of $\Delta(n)$ over the relevant window), a line $f_0(n)$ with slope $1/\alpha$ cannot stay this close to a curve $f(n)$ with slope at least $1/\alpha+\delta$ for more than roughly $2\Delta/\delta$ consecutive values of $n$.

Of course, $\chi(G_n)$ is not deterministic. But in proving our result, we may assume that it is \emph{almost} deterministic: for each $n$, we may assume that $\chi(G_n)$ is concentrated on some interval $[s_n,t_n]$ of length $\ell_n$. With $n'=n+\alpha r$ for $r$ not too large, as before, our coupling implies that $s_{n'}\le t_n+r$, since it is reasonably likely that all inequalities in the chain $s_{n'}\le \chi(G_{n'})\le \chi(G_n)+r\le t_n+r$ hold. In turn, this gives
\[
 s_{n'}\le s_n + \ell_n + r.
\]

Can we still get a contradiction? Yes, if the numbers work out correctly. Defining $f_0(n)=s_n$, we see that between $n$ and $n'=n+\alpha r$ this function has slope
\[
 \frac{s_{n'}-s_n}{\alpha r} \le \frac{1}{\alpha}+\frac{\ell_n}{\alpha r},
\]
so if $\ell_n/(\alpha r)\le \delta/2$, say, we will get a contradiction much as before. Hence there must be some $n$ such that $\ell_n > \alpha r\delta/2$. Note that to obtain a strong non-concentration result, we wish to take $r$ as large as possible.

\subsection{The framework lemma}\label{ss:framework}

In this subsection we formalize the outline above in the following lemma. We have replaced various $o(1)$ bounds here by concrete bounds for definiteness, though in the application we mostly start with $o(1)$ bounds and take $n$ large enough. In the application we will take $a=\alpha(n^-)=\alpha(n^+)$, where $\alpha(n)$ is defined in \eqref{alphadef} and $[n^-,n^+]$ is the range of values of $n$ we consider. Thus $a$ will be the typical independence number of the relevant graphs $\Gnp$.

\begin{lemma}\label{lem:framework}
Let $p$, $\delta$ and $\Delta$ be positive real numbers with $p<1$, and let $n^-<n^+$ and $a$ be positive integers. Let $I=[n^-,n^+]$. Suppose that the following hold. Firstly, there is some function $f(n)$ such that for each (integer) $n\in I$ we have
\begin{equation}\label{fra1}
 \Pr\bb{ \chi(G_{n,p})  \in [f(n)-\Delta,f(n)+\Delta] } \ge 0.99.
\end{equation}
Secondly, for all (real) $n\in I$ we have
\begin{equation}\label{fra2}
 \dfdn \ge \frac{1}{a} + \delta.
\end{equation}
Thirdly, for each $n\in I$ we have
\begin{equation}\label{fra3}
 \Pr\bb{ \chi(G_{n,p}) \in [s_n,t_n] } \ge 0.9
\end{equation}
for some integers $s_n$, $t_n$.
Fourthly, there is an increasing integer-valued function $r(n)$ such that
for each $n\in I$ we have a coupling of $G_{n,p}$ and $G_{n+ar(n),p}$ such that 
\begin{equation}\label{fra4}
 \Pr\bb{ \chi(G_{n+ar(n),p}) \le \chi(G_{n,p}) +r(n)} \ge 0.4.
\end{equation}
Finally, suppose also that
\begin{equation}\label{fra5}
 n^+-n^- \ge 5\Delta/\delta, \text{\quad and\quad} n^+-n^-\ge 5ar(n^+).
\end{equation}
Then there is some integer $n\in I$ for which $t_n-s_n > \frac{a\delta r(n)}{2}$.
\end{lemma}
\begin{proof}
We follow the plan described in the previous section, with the minor complication that we allow $r$ to vary with $n$; this is not essential, but gives stronger results in some applications. Throughout we write $G_n$ for $G_{n,p}$. 

Firstly, for $n\in I$ define
$\ts_n=\max\{s_n,f(n)-\Delta\}$ and $\ttt_n=\min\{t_n,f(n)+\Delta\}$. Then by \eqref{fra1} and \eqref{fra3}
we have
\begin{equation}\label{tstt}
 \Pr\bb{ \chi(G_n) \in [\ts_n,\ttt_n] } \ge 0.89,
\end{equation}
and in particular this interval is non-empty, which implies that
\begin{equation}\label{tsf}
 |\ts_n-f(n)| \le \Delta.
\end{equation}
Let us suppose for a contradiction that for every $n\in I$ we have
\[
 \ell_n := t_n-s_n \le \frac{a\delta r(n)}{2},
\]
and note for later that $\ttt_n-\ts_n\le \ell_n$.
Now, for any $n\in I$ such that $n'=n+ar(n)\in I$, by \eqref{fra4} and \eqref{tstt} (applied twice), with probability at least $0.4-2\times 0.11>0$ all three inequalities $\chi(G_{n'})\ge \ts_{n'}$, $\chi(G_n)\le \ttt_n$ and $\chi(G_{n'})\le \chi(G_n)+r(n)$ hold. Hence, with positive probability
\[
 \ts_{n'}\le \chi(G_{n'}) \le \chi(G_n) + r(n) \le \ttt_n+r(n),
\]
and in particular $\ts_{n'}\le \ttt_n+r(n)$. Since this is a deterministic statement, it always holds.
Thus, recalling that $\ttt_n-\ts_n\le \ell_n$, we have
\begin{equation}\label{tsnar}
 \ts_{n'} \le \ts_n + \ell_n +r(n) \le \ts_n + r(n) + \frac{a\delta r(n)}{2} = \ts_n + (n'-n)\left(\frac{1}{a}+\frac{\delta}{2}\right).
\end{equation}

Finally, define a sequence $(n_i)_{0\le i\le j}$ as follows: let $n_0=n^-$ and, given $n_i$,
let $n_{i+1}=n_i+ar(n_i)$ unless this value exceeds $n^+$, in which case we set $j=i$ and stop.
Note that by the stopping condition and the monotonicity of $r$,
\begin{equation}\label{nstop}
 n_j > n^+ - a r(n_j) \ge n^+ - a r(n^+).
\end{equation}
Applying \eqref{tsnar} with $n=n_i$ (and so $n'=n_{i+1}$) for $0\le i<j$ and telescoping, we see that
\[
 \ts_{n_j} - \ts_{n_0} \le (n_j-n_0) \left(\tfrac{1}{a}+\tfrac{\delta}{2}\right).
\]
On the other hand, from \eqref{fra2}, $f(n_j)-f(n_0) \ge (n_j-n_0) \left(\tfrac{1}{a}+\delta\right)$,
so writing $h(n)=f(n)-\ts_n$ we have
\[
 h(n_j) - h(n_0) \ge (n_j-n_0)\delta/2.
\]
From \eqref{nstop} we have $n_j-n_0> n^+-n^--ar(n^+)$. Hence, by \eqref{fra5},
we have $n_j-n_0 > (4/5)(n^+-n^-) \ge 4\Delta/\delta$. Thus $h(n_j)-h(n_0)> 2\Delta$,
which contradicts \eqref{tsf}.
\end{proof}


\subsection{The coupling argument}\label{ss:coupling}

In this section we present the coupling lemma we shall use. We state it somewhat more generally than needed here; in the application we will take $a=\alpha(n)$ (the typical independence number of $G_{n,p}$).

\begin{lemma} \label{lemma:coupling} 
Let $p \in (0,1)$ be constant, let $b=1/(1-p)$, and let $a=a(n)$ satisfy $ 1.01 \log_b n \le a \le 100 \log_b n$ and $\mu \le n^{1.99}$,
where $\mu=\mu_a(n)=\binom{n}{a}(1-p)^{\binom{a}{2}}$.
Then there is a coupling of the random graphs $G_n=G_{n,p}$ and $G_{n-a}=G_{n-a,p}$
with the property that
\[
 \Pr\bb{ \chi(G_n)\le \chi(G_{n-a})+1 } \ge 1- \frac{1+o(1)}{2\sqrt{\mu}}.
\]
\end{lemma} 
\begin{proof}
Let $U$ be a uniform random subset of $V=[n]$ of size $a$. Given $U$, let $P_n$ be the random graph on $V$ with no edges inside $U$, in which each of the other $\binom{n}{2}-\binom{a}{2}$ possible edges is present independently with probability $p$. 
Thus $P_n$ is $G_n$ with a random independent $a$-set `planted'. From the definition, we may realise $G_{n-a}$ as $P_n[V\setminus U]$.
Furthermore, since $U$ is an independent set in $P_n$, we have
\[
\chi(P_n) \le \chi(G_{n-a}) +1.
\]
It remains only to show that we can couple the distributions of $P_n$ and $G_n$ to agree with
sufficiently high probability.

The key observation is that $P_n$ has the distribution of $G_n$ `size-biased' by the number
$X_a$ of independent $a$-sets. To see this, let $H$ be any graph on $[n]$, let $q_H=\Pr(P_n=H)$, and let $X_a(H)$ be the number of independent $a$-sets in $H$. 
For $P_n=H$ to hold, our random set $U$ must be independent in $H$, which has probability $X_a(H)/\binom{n}{a}$. Given such a choice of $U$, exactly the right edges outside $U$ must be present. Hence
\begin{equation}\label{qpH}
 q_H = X_a(H)\binom{n}{a}^{-1} p^{e(H)}(1-p)^{\binom{n}{2}-\binom{a}{2}-e(H)}
 = \frac{X_a(H)}{\mu} p^{e(H)}(1-p)^{e(H^\cc)} = \frac{X_a(H)}{\mu} p_H,
\end{equation}
where $p_H = \Pr(G_n=H)$.

Let $\tau$ be the total variation distance between the distributions of $P_n$ and of $G_n$.
Then
\[
 2\tau := \sum_H |q_H-p_H| =\sum_H \left|\frac{X_a}{\mu}-1 \right|p_H = 
 \E\left[\frac{|X_a-\mu|}{\mu}\right],
\]
where the expectation refers to the random graph $G_n$. Thus by Jensen's inequality (or by Cauchy--Schwarz),
\[
 4\tau^2 \le  \mu^{-2}\E[ (X_a(G_n)-\mu)^2] = \mu^{-2}\Var[X_a(G_n)].
\]
Writing $\Var[X_a(G_n)]$ as a sum (of covariances of indicator functions) over pairs $U_1$, $U_2$ 
of $a$-sets in $V$, the contribution from $U_1=U_2$ is at most $\mu$, while by a standard
exercise the contribution from the remaining terms is $O(\mu^2a^4/n^2 + \mu a n(1-p)^{a-1})$,
with the two terms corresponding to $U_1$ and $U_2$ intersecting in $2$ or $a-1$ vertices, respectively.
Under our assumptions $\Var[X_a(G_n)]\sim \mu$, so $\tau \lesssim 1 / (2 \sqrt{\mu})$.
Since $G_n$ and $P_n$ can be coupled to agree with probability $1-\tau$, this completes the proof.
\end{proof}

\begin{remark}
It is perhaps interesting that the proof of our coupling lemma relies on a variance bound, i.e., an upper bound on how much $X_a(G_n)$ varies. In the end, we use the lemma to show, roughly speaking, that $\chi(G_n)$ varies \emph{at least} a certain amount, because $X_a(G_n)$ does.
\end{remark}

\begin{corollary}
\label{corollary:coupling} 
Let $p \in (0,1)$ be constant, let $b=1/(1-p)$ and let $1\le a=a(n)\le n$ satisfy $ 1.02 \log_b n \le a \le 99 \log_b n$ and $\mu \le n^{1.98}$,
where $\mu=\mu_a(n)=\binom{n}{a}(1-p)^{\binom{a}{2}}$. Let $r \le \sqrt{\mu}$ be an integer.
Then if $n$ is large enough, there is a coupling of the random graphs $G_n=G_{n,p}$ and $G_{n+ar}=G_{n+ar,p}$
with the property that
\[
 \Pr\bb{ \chi(G_{n+ar})\le \chi(G_{n})+r } > 0.4.
\]
\end{corollary}

\begin{proof}
 For $i= 0, \dots, r$,  let $n_i=n+ai$, and let $\mu_i = \binom{n_i}{a} (1-p)^{\binom{a}{2}}$. In this notation, $n_0=n$, $n_r=n+ar$ and $\mu_0=\mu$. Since $r \le \sqrt{\mu} \le n^{0.99}$ and $a \le 99 \log_b n$, if $n$ is large enough, for all $0 \le i \le r$,
 \[n \le n_i \le n +n^{0.999}.\]
In particular, $\log n_i = \log n + O(n^{-0.001})$, so $1.01 \log_b n_i \le a \le 100 \log_b n_i$ if $n$ is large enough. Furthermore, if $n$ is large enough,
\[
 \mu_i = \mu_0  \left(1+O\left( \frac{ra}{n} \right)\right)^a =  \mu_0 \left(1+O\left( \frac{ra^2}{n} \right)\right) \le n^{1.99} \le n_i^{1.99}.
\]
So we may apply Lemma~\ref{lemma:coupling} to show that, for every $i\in \{1, \dots, r\}$, there is a coupling of the random graphs $G_{n_{i}}$ and $G_{n_{i-1}}$ such that
\begin{equation}\label{eq:couplingi}
 \Pr\bb{ \chi(G_{n_{i}})\le \chi(G_{n_{i-1}})+1 } \ge 1-\frac{1+o(1)}{2\sqrt{\mu_{i}}} \ge 1- \frac{1+o(1)}{2\sqrt{\mu}} .
\end{equation}
The Gluing Lemma (which is trivial in this finite setting\footnote{Given couplings of $X$ and $Y$ and of $Y$ and $Z$, i.e., desired distributions for $(X,Y)$ and for $(Y,Z)$, construct $(X,Y,Z)$ by starting with $Y$ and, given the value of $Y$, taking the appropriate conditional distributions for $X$ and for $Z$ -- for example with conditional independence.}) implies that there is a joint coupling of the random graphs $G_{n_0}, \dots, G_{n_r}$ so that \eqref{eq:couplingi} holds for every $1 \le i \le r$. In this coupling, with probability at least
\[
 1-(1+o(1))\frac{r}{2\sqrt{\mu}} \ge 1-\frac{1+o(1)}{2} > 0.4
\]
we have
$\chi(G_{n+ar})\le \chi(G_{n})+r$.
\end{proof}

\subsection{Preliminaries for Theorem~\ref{theorem:nstar}}\label{ss:prelim}
Fix $p\le 1-1/e^2$, and let
\begin{equation}\label{eq:deff}
  f(n)= f_p(n) := \frac{n}{2 \log_b n - 2 \log_b \log_b n -2 \log_b 2}
 \end{equation}
be the estimate for $\chi(\Gnp)$ given in Theorem~\ref{theorem:bounds}. Note that
\[
 f(n) = \frac{n}{\alpha_0(n)-1 -\frac{2}{\log b}}, 
\]
where $\alpha_0(n)$ was defined in \eqref{eq:a0def}. 
\begin{lemma}\label{lemma:derivative}
Treating $\alpha_0$ and $f$ as functions of a real-valued input $n$, we have
\[\frac{\mathrm{d}f}{\mathrm{d}n} = \frac{1}{\alpha_0(n)}+\frac{1}{\alpha_0(n)^2} + O\left( \frac{1}{\log^3 n}\right).\]
\end{lemma}
\begin{proof}
Elementary calculus!
\end{proof}

Let us note some simple properties of $\alpha_0(n)$, $\alpha(n)=\floor{\alpha_0(n)}$, and $\muexp(n)$, defined in \eqref{eq:mua}. Firstly, as noted in the introduction (see \eqref{rhoalpha1} and Figure~\ref{fig:alphaexponent}),
\begin{equation}\label{rhoalpha}
 \muexp(n) = \alpha_0(n)-\alpha(n)+o(1).
\end{equation}
In other words, $\muexp$ is essentially the fractional part of $\alpha_0$.
Secondly, it is immediate from the definition that $\alpha_0(n)$ is an increasing function of $n$ (for $n$ at least some constant $n_0$), and that
\begin{equation}\label{alphagrow}
 n'\sim n \implies \alpha_0(n')=\alpha_0(n) + o(1).
\end{equation}

As outlined in Section~\ref{ss:framework}, we will want to compare $f'$ to $1/\alpha$. The following lemma is a convenient form of the statement, allowing us to conveniently consider all $n$ in a suitable range.

\begin{lemma}\label{lemma:growthlowerbound}
If $n \sim n'$, then
\[
 \left.\dfdn\right|_{n'} = \frac{1}{\alpha(n)}+\frac{1-\muexp(n)}{\alpha(n)^2} + o \left( \frac{1}{\log^2 n} \right).
\]
\end{lemma}
\begin{proof}
The case $n'=n$ is immediate from Lemma~\ref{lemma:derivative} and~\eqref{rhoalpha}. To see the result for $n'\sim n$, note that when we change $n$ by a factor of $1+o(1)$, from \eqref{alphagrow} the expression for $f'$ given in Lemma~\ref{lemma:derivative} changes by $o(1/\alpha_0^2)=o(1/\log^2 n)$.
\end{proof}

\subsection{Proof of Theorem~\ref{theorem:nstar}}
\label{ss:nstar}

\begin{proof}[Proof of Theorem~\ref{theorem:nstar}]
Throughout we fix $p\le 1-1/e^2$ and $\eps>0$, and consider a positive integer (or rather a sequence) $n$ such that $\mu(n) := \mu_{\alpha(n)}(n) \le n^{1-\eps}$, or, equivalently,
\begin{equation}
 \label{eq:defn1}
 \muexp(n) \le 1-\eps.
\end{equation}
We will find the required $n^*$ if $n$ is large enough.

We will apply Lemma~\ref{lem:framework} with $n^-=n$. Thus from now on we write $n^-$ for our `input' value of $n$.
We choose $\gamma=\gamma(n^-)$ tending to zero sufficiently slowly for various estimates below to hold, and will
choose $n^+$ so that $n^+\le n^-(1+\gamma)$. Thus the desired condition $n^*\sim n^-$ will follow from the conclusion
$n^*\in I=[n^-,n^+]$ of Lemma~\ref{lem:framework}.

As noted above, it is immediate from the definition \eqref{eq:a0def} of $\alpha_0(n)$ that (i) $\alpha_0$ is an increasing function, and (ii) $n^+\sim n^-$ implies $\alpha_0(n^+)=\alpha_0(n^-)+o(1)$. From \eqref{rhoalpha}, the fractional part of $\alpha_0(n^-)$ is $\muexp(n^-)+o(1)$, which is at most $1-\eps/2$, say, if $n^-$ is large enough, which we assume from now on. Thus $\alpha(n^+)=\alpha(n^-)$. In other words, the condition \eqref{eq:defn1} ensures that $\alpha_0(n)$ is not just about to pass through an integer value as we increase $n$ from $n^-$. Let us write
\[
 a = \alpha(n^-),
\]
noting that in fact $\alpha(n)=a$ for all $n\in I$.

Let $f(n)$ and $\Delta(n)$ be as in Theorem~\ref{theorem:bounds}. In particular, the error function $\Delta(n)$ is $o(n/\log^2n)$. We will take
\begin{equation}\label{eq:Delta}
 \Delta=\max_{n\in I}\Delta(n) = o(n^-/\log^2n^-).
\end{equation}
By Lemma~\ref{lemma:growthlowerbound} and \eqref{eq:defn1},
if $n^-$ is large enough we have $f'(n)\ge 1/a+\delta$ for all $n \in I$, 
where
\begin{equation}\label{eq:delta}
 \delta= \frac{\eps}{2a^2} \sim \frac{\eps}{8\log_b^2n^-}.
\end{equation}
So far, we have verified the first two conditions of Lemma~\ref{lem:framework}. For the third, by assumption we have $\Pr\bb{ \chi(G_{n,p}) \in [s_n,t_n] } \ge 0.9$, and our aim is to prove a lower bound on some $\ell_n:=t_n-s_n$.

For \eqref{fra4}, we take $r(n)=  \lfloor \sqrt{\mu(n)}  \rfloor$
where $\mu(n)=\mu_a(n)=\binom{n}{a}(1-p)^{\binom{a}{2}}$ as usual; note that here $a=\alpha(n)$. This is clearly an increasing function of $n$. Moreover, since $a\sim 2\log_b n$ the first condition of Corollary~\ref{corollary:coupling} holds with room to spare. For the second, for any $n\in I$ we have $\mu(n)=n^{\muexp(n)}$ by definition, and from \eqref{rhoalpha} and \eqref{alphagrow} we have $\muexp(n)=\muexp(n^-)+o(1)\le 1$, so $\mu(n)\le n\le n^{1.98}$. Hence Corollary~\ref{corollary:coupling} applies, establishing \eqref{fra4}.

Finally, from \eqref{eq:Delta} and \eqref{eq:delta} we have $\Delta/\delta=o(n^-)$.
Also,
$r(n^+) \le \sqrt{\mu(n^+)}\le \sqrt{n^+}$ as above, so both lower bounds on $n^+-n^-$ in \eqref{fra5} are $o(n^-)$, and we can choose $n^+$ to satisfy these bounds as long as $\gamma(n)\to 0$ slowly enough.

Thus, all conditions of Lemma~\ref{lem:framework} are met, and we conclude that there is some $n^*\in I$ such that
\[
 t_{n^*}-s_{n^*}\ge \frac{a\delta}{2} r(n^*) \sim \frac{\eps}{4\alpha(n^*)} \sqrt{\mu(n^*)}
 \sim \frac{\eps}{8\log_b n^*}\sqrt{\mu(n^*)} = \frac{\eps\log b}{8\log n^*}\sqrt{\mu(n^*)}.
\]
This establishes the conclusion of Theorem~\ref{theorem:nstar} if $n$ is large enough.
\end{proof}

\begin{remark}\label{rem:bottleneck}
Let us comment briefly on how the error bound in Theorem~\ref{theorem:bounds} affects the final bounds we obtain.
At first sight, it appears to play little role: the interval length we obtain depends on $\delta$ (the gradient difference) and $r=\sqrt{\mu(n)}$. However, via \eqref{fra5}, if $\Delta$ is large we need to consider a large range $I$ of possible values of $n$. This not only weakens the conclusion (finding $n^*$ far from $n$) but can cause a more serious problem: over the interval $I$ both $\mu(n)$ and $\delta(n)=f'(n)-1/\alpha(n)$ vary, so if our bound on $\Delta$ is too weak, we will not obtain a useful lower bound on $\delta$ and the argument will fail. Conversely, to obtain a final non-concentration length very close to $n^{1/2}$, we need to consider values of $n$ such that $\mu(n)$ is very close to $n$, which will only be true over a relatively short interval. So for this we need a better bound on $\Delta$. We revisit this in Section~\ref{s:polylog}.
\end{remark}

\subsection{Proof of Theorem~\ref{theorem:n12}}\label{ss:n12}
Fix $c<1/2$, and suppose that $[s_n, t_n]$ is a sequence of intervals which contains $\chi(G_{n,p})$ whp, with interval lengths $\ell_n:=t_n-s_n$. We will show that there is an integer $n^*$ such that $\ell_{n^*} \ge \left(n^*\right)^c$, which suffices to prove Theorem~\ref{theorem:n12}.

Let
\begin{equation}\label{eq:epsdef}
 \eps = \frac{1-2c}{3}  \in \left(0, \frac 13 \right).
\end{equation}
It is very easy to see that we can find an arbitrarily large integer $n$ such that
\begin{equation}\label{rhochoice}
 \muexp(n) \in \left(1-2\eps, 1-\eps\right).
\end{equation}
(Recall from \eqref{rhoalpha} that $\muexp$ is essentially the fractional part of $\alpha_0(n)$, which increases smoothly with $n$.)
By definition of $\muexp(n)$, this implies that
\begin{equation}
 \mu(n) =n^{\muexp(n)} \in \left( n^{1-2\eps}, n^{1-\eps}\right). \label{eq:r1}
\end{equation}
We first consider the case $p \le 1-1/e^2$; the case $p>1-1/e^2$ will follow by some straightforward modifications which we describe at the end of the proof. By Theorem~\ref{theorem:nstar}, there is an integer $n^* \sim n$ such that
\begin{equation}\label{eq:elln}
 \ell_{n^*} \ge C(\eps,p) \frac{\sqrt{\mu(n^*)}}{\log n^*}
 \end{equation}
As $n^*\sim n$, it follows that $\alpha_0(n^*) = \alpha_0(n)+o(1)$, and so (by \eqref{rhoalpha} and \eqref{rhochoice}) $\alpha(n^*)=\alpha(n)$. Therefore,
\[
 \mu(n^*) \sim  \mu(n) \left(n^*/n \right)^{\alpha(n)} = \mu(n)(1+o(1))^{O(\log n)} = \mu(n) n^{o(1)} \ge n^{1-2\eps+o(1)} = \left(n^*\right)^{1-2\eps+o(1)}.
\]
From \eqref{eq:elln} it follows that if $n$ is large enough, then
\[
 \ell_{n^*} \ge \left(n^*\right)^{\frac{1-2\eps+o(1)}{2}} >   \left(n^*\right)^{\frac{1-3\eps}{2}} = \left(n^*\right)^{c}, 
\]
as required.

Now suppose that $p > 1-1/e^2$. So far in this paper, whenever we assumed $p\le 1-1/e^2$, it was only to be able to use the estimate for the chromatic number from Theorem \ref{theorem:bounds}. More specifically, we only used that we have some estimate $\chi(G_{n,p})=f(n)+o \left(\frac{n}{\log^2 n} \right)$ so that the derivative $f'(n)$ is sufficiently larger than $\tfrac 1\alpha$; namely that 
\begin{equation}\label{eq:whatweused}
 f'(n) \ge \frac{1}{\alpha(n)} + \frac{1-\muexp}{\alpha(n)^2}+o \left(\frac{1}{\log^2 n}  \right).
\end{equation}
If $p>1-1/e^2$, \cite{heckel2018chromatic} gives a more complicated expression which also determines $\chi(G_{n,p})$ up to an error term of size $o \left( \frac {n} {\log^2n}\right)$. Fortunately, if $\muexp = \alpha_0 - \alpha +o(1)$ is close to $1$, this estimate takes a simple form which is given in the following lemma.
\begin{lemma}\label{lem:estimatelargep}
 Fix $p> 1-1/e^2$, and let $u= \frac{2}{\log b}< 1$. For all $n$ such that $\alpha_0(n) - \alpha(n) \ge u$, whp
 \[
  \chi(G_{n,p}) = \frac{n}{\alpha(n)-1} + o\left(\frac{n}{\log^ 2n}\right).
 \]
\end{lemma}
\begin{proof}
By Theorem~1 in \cite{heckel2018chromatic}, whp
 \begin{equation}\label{eq:estimateplarge}
  \chi(G_{n,p}) = \frac{n}{\gamma(n) - x_0} + o\left(\frac{n}{\log^ 2n}\right),
 \end{equation}
where $\gamma(n) = \alpha_0(n) -1-u$, and, letting $d = \gamma - \left \lfloor \gamma \right \rfloor$, $x_0$ is the smallest non-negative solution to
\[
 \varphi(x) := (1-d+x) \log (1-d+x) + (d-x)(1-d)/u \le 0.
\] 
Suppose that $\alpha_0(n) - \alpha(n) \ge u$. Then $\left \lfloor \gamma \right \rfloor = \left \lfloor \alpha_0 - 1 - u \right \rfloor = \alpha - 1$ and $d = \gamma - \left \lfloor \gamma \right \rfloor = \alpha_0 - \alpha - u$. In particular, $u < 1- d$. 
Then for $0\le x \le d$, note that 
\begin{align*}
\varphi'(x) &= \log (1-d+x) + 1 - (1-d)/u  \le  1 - (1-d)/u  < 0 .
\end{align*}
As $\varphi(d)=0$, this implies that $d$ is the smallest nonnegative solution to $\varphi(x) \le 0$, and so $x_0=d$. By \eqref{eq:estimateplarge}, whp
 \begin{equation*}
  \chi(G_{n,p}) = \frac{n}{\gamma - d} + o\left(\frac{n}{\log^ 2n}\right) = \frac{n}{\left \lfloor \gamma \right \rfloor} + o\left(\frac{n}{\log^ 2n}\right) = \frac{n}{\alpha-1} + o\left(\frac{n}{\log^ 2n}\right).
 \end{equation*}
\end{proof}
Let $u = \tfrac 2 {\log b}$, and fix $\eps>0$. 
For $n$ large enough, if $\muexp  > u+ \eps  $  then $\alpha_0 - \alpha = \muexp+o(1) \ge u$ and Lemma~\ref{lem:estimatelargep} above applies. Let $f(n) = \frac{n}{\alpha(n)-1}$. If we only consider $n$ in an interval where $\alpha(n)$ is constant --- as we did the proof of Theorem~\ref{theorem:nstar} --- we have
\[
 f'(n) = \frac{1}{\alpha-1}
  = \frac{1}{\alpha} + \frac{1}{\alpha(\alpha-1)} 
 > \frac{1}{\alpha} + \frac{1}{\alpha^2} 
 \ge  \frac{1}{\alpha}+\frac{1-\muexp}{\alpha^2}.
\]
Comparing this to \eqref{eq:whatweused}, all our conclusions from the case $p\le 1-1/e^2$ remain valid as long as $\muexp \in (u+\eps,1)$. To prove the statement of Theorem~\ref{theorem:n12}, we can assume $c$ is arbitrarily close to $\tfrac 12$, so by \eqref{eq:epsdef} we can make $\eps$ arbitrarily small. By \eqref{rhochoice}, we can assume that $\muexp \in (u+\eps,1)$. The rest of the proof of Theorem~\ref{theorem:n12} is unchanged from the case $p\le 1-1/e^2$.
\qed

\section{Proof of Theorem~\ref{theorem:polylog}}\label{s:polylog}

In this section we prove our final result, Theorem~\ref{theorem:polylog}. Throughout, we fix $p=\tfrac 12$. When we use our assumption \eqref{eq:assump}, we shall state this explicitly. This happens only at one point in the proof of Theorem~\ref{theorem:polylog}; the assumption is not needed for any of our lemmas. The overall proof strategy is very similar to that we used for Theorem~\ref{theorem:nstar}, based on our Framework Lemma, Lemma~\ref{lem:framework}. Before turning to the details, let us outline roughly why we choose the parameters that we do, as motivation for the arguments that follow.

We use the same coupling lemma as before which, in terms of the parameters of Lemma~\ref{lem:framework}, leads to choosing $r\approx \sqrt{\mu_\alpha(n)}$, where $\alpha=\alpha(n)$. As Lemma~\ref{lem:framework} produces an interval of length at least $\alpha \delta r / 2$, for a given value of $\alpha$ we want $\delta r$ to be as large as possible, so we try to choose $n$ so that $\delta(n) \sqrt{\mu_\alpha(n)}$ is as large as possible, which turns out to be when $\mu_\alpha(n) \approx n/\log^2 n$. For why this is optimal, see \S\ref{s:intuition}.

As we shall see below, in this range the difference $\delta$ between the slope of the chromatic number and $1/\alpha$ is quite small, of order $\Theta(\log\log n/\log^3 n)$. We will consider an interval of values of $n$ differing by at most a factor $1+x$ where $x\approx 1/\log n$, so that, over the range of $n$, $\mu_\alpha$, which is roughly proportional to $n^\alpha$, varies by a constant factor. This means that we need our $\Delta$ to be at most roughly $xn\delta\approx n \log\log n/\log^4 n$, to satisfy the first condition in \eqref{fra5}.

The error bound from the concentration result \eqref{eq:assump} is much smaller than this. The trouble is that it applies to the $\beta$-bounded chromatic number $\chi_\beta$, where $\beta=\alpha-1$, not the chromatic number itself. However, it turns out that, by a first moment argument, we can bound $\chi$ from below by $k_\beta-O(\mu_\alpha(n)\log\log n/\log^2 n)$, where $k_\beta$ is the first moment threshold for $\beta$-bounded colourings; see Definition~\ref{def:tbdd}. Since \eqref{eq:assump} gives $\chi_\beta\le k_\beta+o(n\log\log n/\log^4 n)$ and $\chi\le\chi_\beta$ by definition, we thus have that $\chi(\Gnh)$ is (just) close enough to $k_\beta$ for our argument to work.

Throughout the section we consider (sometimes only integer, sometimes real) values of $n$ in a set $W\subset \R$ with the following property: $W$ is a disjoint union of intervals, on each of which $\alpha(n)$ is constant, where $\alpha(n)$ is defined in \eqref{eq:a0def} and \eqref{alphadef}. In short, $\alpha(n)$ is \emph{locally constant} on $W$, formally meaning that it has derivative zero.

In the following arguments, there are two relevant ways that $n$ varies: within an interval, and between intervals. When we differentiate with respect to $n$, we are (by definition) working locally within an interval, and then $\alpha=\alpha(n)$ is constant. On the other hand, for asymptotics (such as the bound $\alpha(n)=O(\log n))$, the variation between intervals is relevant. 

Intuitively, one can think of $n$ as very large (so that various asymptotic estimates hold), and in the analysis, in particular the application of the framework lemma, it is only the variation within an interval that matters. So one should think of $\alpha(n)$ as a constant (derivative zero) that happens to be of logarithmic order. Formally, of course, there is no issue: $\frac{\dd \alpha}{\dd n}=0$ on the set $W$.

The hardest part of the proof turns out to be understanding the behaviour of (a suitable approximation to) $k_\beta(n)$, where $\beta=\beta(n)=\alpha(n)-1$. The following lemma, proved in the next section, provides this. Note that we work almost all the time with $\alpha-1$ rather than $\alpha$, so to keep the formulae compact we write $\beta$ for $\alpha-1$. In fact, although we don't need it here, the same method works with no difficulty for $\alpha-2$ also. We prove the more general case since it may be useful elsewhere, but the reader may wish to simply consider $\beta=\alpha-1$.

\begin{lemma}\label{lem:polylog}
Define $\beta=\beta(n)=\alpha(n)-i$, where $\alpha(n)$ is defined in \eqref{alphadef} and $i\in \{1,2\}$ is constant,
and let $W\subset \R$ consist of a disjoint union of intervals on each of which $\alpha(n)$ is constant. For \emph{integer} $n\in W$, define $k_\beta(n)$ as in \eqref{ktdef}. Then there is a real-valued function $k^*(n)=k^*_\beta(n)$, defined for all $n\in W$, with the following three properties:
\[
 k_\beta(n) = k^*(n) + O(\log^2 n) \text{ for all integer }n\in W,
\]
while for all real $n\in W$ we have
\begin{equation}\label{noks}
 \frac{n}{k^*(n)} = \alpha(n) + \frac{\log(\mu_{\alpha(n)}(n))}{\log n-\log\log n} - \frac{2}{\log 2} - 1 + O(1/\log n),
\end{equation}
and
\[
 \left(\frac{\dd k^*(n)}{\dd n}\right)^{-1} = \frac{n}{k^*(n)} + \frac{2}{\log 2} + O(1/\log n).
\]
\end{lemma}
The (somewhat lengthy) proof of Lemma~\ref{lem:polylog} is given in Section~\ref{ss:polylog2}.
\begin{remark}\label{rem:noks}
The formula \eqref{noks} may seem slightly mysterious; we make two observations. Firstly, $\alpha(n)$ here can be replaced by any integer $a(n)$ such that $a(n)=\alpha_0(n)+O(1)$, provided we use the same $a$ in both places. This follows from the fact that $\mu_a/\mu_{a-1}=\Theta(\log n/n)$ for $a=\alpha_0(n)+O(1)$. Secondly, a straightforward but rather tedious calculation shows that, for $a=\alpha_0(n)+O(1)$, we have
\begin{equation}\label{noks0}
 a(n) + \frac{\log(\mu_{a(n)}(n))}{\log n-\log\log n}
 = \alpha_0(n) + \left(\frac{2}{\log 2}-\frac12\right)\frac{\log\log n}{\log n} + O(1/\log n).
\end{equation}
To make sense of this note that one can interpolate the definition of $\mu_a(n)$ to non-integer values of $a$ in a natural way. As noted above, the left-hand side is then (roughly) constant for $a$ near $\alpha_0$. Substituting in $a=\alpha_0$, we expect $\mu_{\alpha_0}$ to be close to $1$. This explains \eqref{noks0} apart from the $c\log\log n/\log n$ term. This term is only there because we have taken a simple definition of $\alpha_0$, rather than solve $\mu_{\alpha_0}=1$ very precisely; if we were to do so, we would simply have $\alpha_0(n)+O(1/\log n)$ here, but there would be minor additional complications in other formulae. Finally, we don't use this expression in \eqref{noks} because both in the proof and in the application, it is easier to work with $\alpha$ and $\mu_\alpha$ than with $\alpha_0$.
\end{remark}

Our next lemma, proved in \S\ref{ss:alphashift}, is the promised lower bound on $\chi(\Gnh)$ in terms of $k_\beta$.

\begin{lemma}\label{lem:as1}
Let $\eps>0$ be constant. Suppose that $\mu_{\alpha(n)}(n)=\Theta(n/\log^2 n)$. Then, whp,
\[
 \chi(\Gnh) \ge k_\beta(n) - (1+\eps)\mu_{\alpha(n)}(n)\frac{\log\log n}{c_0\log^2 n},
\]
where $c_0=2/\log 2$, and $k_\beta(n)$ is defined in \eqref{ktdef}, with $\beta=\beta(n)=\alpha(n)-1$.
\end{lemma}

At this point we are ready to prove Theorem~\ref{theorem:polylog}, subject to the (in the first case lengthy) proofs of Lemmas~\ref{lem:polylog} and~\ref{lem:as1}, given in next two sections.

\begin{proof}[Proof of Theorem~\ref{theorem:polylog}]
Set $c_0=2/\log 2$, and let $c_1$ be a positive constant with $c_1<\frac{1}{5c_0^2}$. We consider the set
\begin{equation}\label{Wdef}
 W := \left\{ n\in \R\ :\  \frac{c_1n}{e\log^2 n} \le \mu_{\alpha(n)}(n) \le \frac{c_1n}{\log^2 n} \right\}.
\end{equation}
This set
is easily seen to be a disjoint union of intervals, one for each value of $\alpha(n)$. Our aim is to show the existence of at least one $n$ in each interval (apart perhaps from the first few) such that $\chi(\Gnh)$ is not too concentrated.

First, we consider the length of a single interval $I=[n^-,n^+]\subset W$. With $a=\alpha(n)$ constant (as it is over $I$), $\mu_a(n)$ is proportional to $\binom{n}{a}$, which is asymptotically $n^a/a!$. It follows easily that $n^+=(1+\eps)n^-$ for some $\eps$ such that $(1+\eps)^a\sim e$. This gives $\eps\sim 1/a\sim 1/(c_0\log n^-)$, say. Thus
\[
  n^+-n^-\sim n^-/(c_0\log n^-).
\]

We will apply Lemma~\ref{lem:framework} to each interval, with $f(n)=k^*(n)$, where $k^*(n)$ is as in Lemma~\ref{lem:polylog}. Let $\beta=\beta(n)=\alpha(n)-1$, which is constant on each interval.
By Lemma~\ref{lem:as1} and the definition of $W$, whp we have
\[
 \chi(\Gnh) \ge k_\beta(n) - (1+o(1))\mu_{\alpha(n)}(n)\frac{\log\log n}{c_0\log^2 n} \ge
 k_\beta(n) - (1+o(1)) \frac{c_1 n\log\log n}{c_0\log^4 n}.
\]
We have $\chi_\beta(\Gnh) = k_\beta(n) + o(n\log\log n/\log^4 n)$ whp by our assumption \eqref{eq:assump}.\footnote{This is the only place in the proof where we use \eqref{eq:assump}; the lemmas stated in this section do not rely on it.}
Thus, whp
\[
 k_\beta(n) - (1+o(1)) \frac{c_1 n\log\log n}{c_0\log^4 n} \le \chi(\Gnh) \le
 \chi_\beta(\Gnh) \le k_\beta(n)+\frac{c_1 n\log\log n}{c_0\log^4 n}.
\]
From Lemma~\ref{lem:polylog} we have $k^*(n)-k_\beta(n)=O(\log^2 n)$ for integer $n\in W$, so it follows that for $n\in I$ we have $\chi(\Gnh)\in [f(n)-\Delta,f(n)+\Delta]$ whp, for some $\Delta$ satisfying
\[
 \Delta \sim \frac{c_1 n^-\log\log n^-}{c_0\log^4 n^-}.
\]
This establishes the first condition \eqref{fra1} of Lemma~\ref{lem:framework}.\footnote{The reader may wonder why we take $f(n)=k^*(n)$ rather than $f(n)=k_\beta(n)$. The reason is that we do not know precisely enough how the latter varies.}

We set $a=\alpha(n)$, which is constant over the interval we are considering. For $n\in W$, by Lemma~\ref{lem:polylog} we have
\begin{eqnarray*}
 \frac{1}{f'(n)} &=& \alpha(n) + \frac{\log(\mu_{\alpha(n)}(n))}{\log n-\log\log n} - 1 + O(1/\log n) \\
 &=& a + \frac{\log n-2\log\log n+O(1)}{\log n-\log\log n} -1 + O(1/\log n) \\
 &=& a - \frac{\log\log n}{\log n} + O(1/\log n).
\end{eqnarray*}
Since $a\sim c_0\log n$, it follows (using $(a-\eps)^{-1}=a^{-1}(1-\eps/a)^{-1} = a^{-1}+\eps a^{-2}+\cdots$) that
\[
 f'(n) = \frac{1}{a} + (1+o(1))\frac{\log\log n}{c_0^2\log^3 n},
\]
so \eqref{fra2} holds for all $n\in I$ for some $\delta$ satisfying
\[
 \delta\sim \frac{\log\log n^-}{c_0^2\log^3 n^-}.
\]

As usual \eqref{fra3} is part of our assumption; we assume $\Gnh$ is concentrated like this and our aim is to give a lower bound on $t_n-s_n$ for some $n$.

As before, condition \eqref{fra4} follows from our coupling result, Corollary~\ref{corollary:coupling}, taking $r(n)=\floor{\sqrt{\mu_a(n)}}$, say.

Now $\Delta/\delta\sim c_0c_1n^-/\log n^-$, while $ar(n^+)$ is $O(\sqrt{n^-})$. By choice of $c_1$ we have $1/c_0>5c_0c_1$, so it follows that the inequalities in \eqref{fra5} hold for large enough $n$.

Thus Lemma~\ref{lem:framework} implies that for some $n$ in each interval (except perhaps for the first $O(1)$), we have
\[
 t_n-s_n\ge \frac{a\delta r(n)}{2} \ge (1+o(1)) c_0\log n \frac{\log\log n}{c_0^2\log^3 n} \frac{\sqrt{c_1 n}}{2\sqrt{e}\log n} \sim c_2\frac{\sqrt{n}\log\log n}{\log^3 n}
\]
where $c_2=\sqrt{c_1}/(2c_0\sqrt{e})$, and we have replaced $n^-$ by $n$ since $n\sim n^-$.
\end{proof}

\subsection{Proof of Lemma~\ref{lem:polylog}}\label{ss:polylog2}

In this section we prove Lemma~\ref{lem:polylog}. This will take some time. In principle, this is a matter of calculation, but it seems to require considerable work, and several tricks, to get the calculations to come out to the required accuracy. For the reader to refer back to later, we collect in Table~\ref{t1} some notation used in this and the next section.

\begin{table}[htb]
\begin{center}
\begin{tabular}{|c|l|}
 \hline
$E_{n,k,t}$ & Expected number of unordered $t$-bounded $k$-colourings. \\
$k_t(n)$ & Threshold where $E_{n,k,t}$ reaches $1$. \\
 \hline
$L_0(n,k,t)$ & approximation to $\log(E_{n,k,t})$ defined in \eqref{L0nktdef}\\
$\hL_0(n,k,t)$ & $\tfrac{1}{k}L_0(n,k,t)$ \\
$\tL_0(\rho,k,t)$ & defined by $\hL_0(n,k,t)=\tL_0(n/k,k,t)$ \\
 \hline
$k^*(n)$ & defined by solving $L_0=0$ (or $\hL_0=0$) \\
 \hline
\end{tabular}
\end{center}
\caption{The various functions involved in defining and approximating the $t$-bounded expectation threshold. While $L_0$ is the key approximation to $\log(E_{n,k,t})$, in different parts of the analysis it turns out to be much simpler to consider the transformed functions $\hL_0$ and $\tL_0$}
\label{t1}
\end{table}

The $t$-bounded first moment threshold $k_t(n)$ is defined in terms of $E_{n,k,t}$, the expected number of unordered $t$-bounded $k$-colourings of $\Gnh$. One key idea of the proof is to replace $E_{n,k,t}$ by a simpler quantity, and to define $k^*$ as the threshold for this simpler estimate to cross $1$. We will simplify in three simple steps, proved together in one lemma (Lemma~\ref{lemsimp} below): (i) we replace the expected number of colourings with a given profile (see below) by a simpler formula, (ii) we replace the sum over profiles by a maximum, and (iii) we replace the maximum over integer-valued profiles (a complicated set) by the maximum over a certain region in $\R^t$.

To state and prove the lemma we need some notation. Let $\pi=(n_i)_{i=1}^{t}$ denote a $t$-bounded profile, where $n_i$ represents the number of colour classes with $i$ vertices. Let $P_{n,k,t}$ denote the set of all profiles $\pi$ satisfying
\begin{equation}\label{constraint}
 n_i\ge 0,\qquad  \sum_{i=1}^t n_i=k  \qquad \text{and}\qquad \sum_{i=1}^t i n_i = n.
\end{equation}
Thus $P_{n,k,t}$ consists of all profiles corresponding to $t$-bounded $k$-colourings. Extending to real values, given positive reals $k<n$ and a positive integer $t$, let
\[
 P^0_{n,k,t} = \bigl\{\  (n_i)_{i=1}^{t} \in \R^t : \text{\eqref{constraint} holds}\ \bigr\}.
\]
Two key quantities appearing in many places in our calculation will be
\begin{equation}\label{didef}
  d_i := 2^{\binom{i}{2}}i!
\end{equation}
and
\begin{equation}\label{L0nktdef}
 L_0(n,k,t):=\sup_{\pi \in P^0_{n,k,t}}\left\{ n\log n -n +k -\sum_{i=1}^t n_i\log(n_id_i) \right\}.
\end{equation}
As we now show, the latter is a good approximation to $\log(E_{n,k,t})$, where $E_{n,k,t}$ is the expected number of unordered $t$-bounded $k$-colourings of $\Gnh$ (see Definition~\ref{def:tbdd}).

\begin{lemma}\label{lemsimp}
Suppose that $t=t(n)=O(\log n)$. For all (large enough) $n$ and for all $k$ with $1<n/k<t$ 
we have
\[
 \log(E_{n,k,t}) = L_0(n,k,t) + O(\log^4 n).
\]
\end{lemma}

\begin{proof}
For a given profile $\pi$, let $E_{\pi}$ be the expected number of unordered colourings with this profile, so by definition
\[
 E_{n,k,t} = \sum_{\pi\in P_{n,k,t}} E_{\pi}.
\]
Since the order of the parts does not matter, there are
\[
 \frac{1}{\prod_i n_i!} \frac{n!}{\prod_i i!^{n_i}}
\]
ways to partition $[n]$ into $k$ parts with $n_i$ of size $i$ for each $i$ (the second fraction is the relevant multinomial coefficient). Such a partition is indeed an unordered $k$-colouring if and only if there are no edges of $\Gnh$ within the parts. Hence
\[
 E_{\pi} =  \frac{n!}{\prod_i n_i!} \frac{1}{\prod_i i!^{n_i}} 2^{-\sum_i n_i \binom{i}{2}}  =  \frac{n!}{\prod_i n_i!} \prod_i d_i^{-n_i}.
\]

Let
\[
 L_\pi := n\log n - \sum_i n_i \log(n_id_i) -n +k.
\]
Then using Stirling's formula it is easy to see that for any $\pi\in P_{n,k,t}$ we have
\[
 \log(E_{\pi}) = L_{\pi}+O(\log^2 n).
\]
Indeed, this follows by absorbing the (logarithm of) all $\sqrt{2\pi m}$ factors into the error term.

There are at most $(n+1)^t=\exp(O(\log^2n))$ possible profiles, so $E_{n,k,t}$ is within this factor of $\max_\pi E_{\pi}$. Hence
\begin{equation}\label{lnEb1}
 \log(E_{n,k,t}) = \max_{\pi\in P_{n,k,t}} L_\pi + O(\log^2 n).
\end{equation}

Now $P_{n,k,t}\subset P^0_{n,k,t}$, so the inequality $\max_{\pi\in P_{n,k,t}} L_\pi\le L_0(n,k,t)$ holds trivially. It remains to show the reverse inequality, up to a small error term. For this, let $\pi=(n_i)_{i=1}^t\in P^0_{n,k,t}$ be arbitrary. Our aim is to find a profile $\pi'\in P_{n,k,t}$ with $L_{\pi'}$ not too far from $L_\pi$.
To do so, we modify $\pi$ in a series of small steps. Firstly, round each (non-integer) $n_i$ either up or down to the nearest integer, choosing whether to round up or down in such a way that after all such roundings $\sum n_i$ is unchanged. At this point, $\sum i n_i$ has changed by no more than $\sum_{i=1}^t i\le t^2$. We obtain $\pi'$ by making a number of further changes, each of which consists of altering the size of one class by $1$, i.e., decreasing some $n_i$ by $1$ and increasing either $n_{i-1}$ or $n_{i+1}$ by $1$; clearly we can fix the error in $\sum in_i$ by at most $t^2$ such changes. In total, we have made $O(t^2)$ small changes, each of which consists of altering a single value $n_i$ by at most $1$. 

Now each $d_i$ is at most $2^{t^2}t! = \exp(O(t^2))$. Also
\[
 \frac{\dd}{\dd n_i} n_i\log(n_id_i) = \log(n_id_i)+1,
\]
which is thus $O(\log n+t^2)=O(\log^2 n)$ for $1\le n_i\le n$. It is easy to check that $n_i\log(n_id_i)$ is $O(\log^2 n)$ for $0\le n_i\le 1$. It follows that each of the changes above (changing a single $n_i$ by at most $1$) changes $\sum n_i\log(n_id_i)$ by at most $O(\log^2 n)$. The remaining terms in $L_\pi$ are the same for $\pi'$ as for $\pi$, so we conclude that
\[
 |L_\pi-L_{\pi'}| = O(t^2 \log^2 n) = O(\log^4 n).
\]
Hence $L_0$ is within $O(\log^4 n)$ of the maximum over (integer) profiles $\pi$, which, combined with \eqref{lnEb1}, gives the result.
\end{proof}

At this point it will be convenient to rescale in two ways: we replace each $n_i$ by $p_i=n_i/k$, the fraction of colour-classes having size $i$ (at least, this is the interpretation when $n_i$ is an integer). We will also divide the logarithm we are considering by $k$. To formalize this, for $t$ a positive integer and $\rho$ a real number with $1<\rho<t$ define
\[
 \tP_{\rho,t} = \left\{\ (p_i)_{i=1}^t \in [0,1]^t\ :\ \sum_i p_i = 1\text{\quad and\quad} \sum_i i p_i = \rho\ \right\}.
\]
When $\rho=n/k$, this is exactly the set $P^0_{n,k,t}$ rescaled by replacing each $n_i$ by $p_i=n_i/k$.
Note that $\tP_{\rho,t}$ is simply the set of probability distributions (or probability mass functions) on $[t]$ with expectation $\rho$.

Let $t$ be a positive integer, and $\rho$ and $k$ positive reals with $1<\rho<t$. For $\bp=(p_i)_{i=1}^t \in \tP_{\rho,t}$, let
\begin{equation}\label{tLdef}
 \tL(\rho,k,\bp) := \rho\log(\rho k)-\log k -\rho +1 -\sum_i p_i\log(p_id_i), 
\end{equation}
and define
\begin{equation}\label{tL0def}
 \tL_0(\rho,k,t) := \sup_{\bp \in \tP_{\rho,t}} \tL(\rho,k,\bp).
\end{equation}
\begin{lemma}\label{lscale}
If $t$ is a positive integer and $n$ and $k$ are positive reals with $1<n/k<t$ 
then
\[
 L_0(n,k,t) = k \tL_0(\rho,k,t),
\]
where $\rho=n/k$.
\end{lemma}
\begin{proof}
This is simply a matter of rescaling: for $\pi\in P^0_{n,k,t}$, letting $p_i=n_i/k$ we have
\begin{multline*}
 \frac{L_\pi}{k} = \frac{n}{k}\log n -\sum_i \frac{n_i}{k}\log(n_id_i) -\frac{n}{k}+1
  = \rho\log(\rho k)-\sum_i p_i\log(kp_id_i) - \rho + 1  \\
 = \rho\log(\rho k)-\sum_i p_i\log(p_id_i) -\log k - \rho + 1,
\end{multline*}
since $\sum p_i=1$. The result follows from the bijection between $P^0_{n,k,t}$ and $\tP_{\rho,t}$ given by $p_i=n_i/k$.
\end{proof}

\begin{corollary}\label{cEnk}
Suppose that $t=t(n)=O(\log n)$. For all (large enough) $n$ and for all $k$ with $1<k<n/t$ we have
\[
 \log(E_{n,k,t}) = k \tL_0(\rho,k,t) + O(\log^4 n),
\]
where $\rho=n/k$.
\end{corollary}
\begin{proof}
Immediate from Lemmas~\ref{lemsimp} and~\ref{lscale}.
\end{proof}

In the next few lemmas our aim is to study the functions $\tL(\rho,k,\bp)$ and $\tL_0(\rho,k,t)$ defined in \eqref{tLdef} and \eqref{tL0def}. Although, as in Corollary~\ref{cEnk}, we will eventually evaluate $\tL_0$ at $(n/k,k,t)$, where $n$ and $k$ are integers, for the moment this is irrelevant. We are simply studying the functions defined in \eqref{tLdef} and \eqref{tL0def}, where $\rho$ and $k$ are real inputs, $t$ is an integer, and $\bp$ is a vector in $\tP_{\rho,t}$. In particular, $n$ appears nowhere in these definitions.

We start by studying the location and value of the maximum of $\tL(\rho,k,\bp)$ over $\bp$.
\begin{lemma}\label{lem:tLmax}
Let $1<\rho<t$, where $t$ is an integer. Then, for any real $k>1$, there is a unique $\bp=\bp_{\rho,t}\in \tP_{\rho,t}$ maximizing $\tL(\rho,k,\bp)$. This maximizing $\bp$ is independent of $k$, and is given by
\begin{equation}\label{pixy}
 p_i = e^{x+iy} d_i^{-1}
\end{equation}
for $1\le i\le t$, where $x=x_t(\rho)$ and $y=y_t(\rho)$ satisfy
\begin{equation}\label{xy1}
 \sum_{i=1}^t e^{x+iy} d_i^{-1} =1
\end{equation}
and
\begin{equation}\label{xy2}
 \sum_{i=1}^t i e^{x+iy} d_i^{-1} = \rho.
\end{equation}
Furthermore,
\begin{equation}\label{tLxy}
 \tL_0(\rho,k,t) = \rho\log(\rho k)-\log k-\rho+1-x-\rho y.
\end{equation}
\end{lemma}
\begin{proof}
Throughout the proof $k$, $t$ and $\rho$ are fixed, and we are maximizing only over $\bp\in \tP_{\rho,t}$. Thus, the only term in $\tL(\rho,k,\bp)$ that varies is the term 
\[
 f(\bp) = \sum_{i=1}^t -p_i\log(p_id_i).
\]
Note that $k$ does not appear in this expression. In contrast, $\rho$ appears implicitly via the constraint $\sum ip_i=\rho$. Hence the location of the maximum will depend on $\rho$ and $t$, but not on~$k$.

Now $-x\log(xd)$ is strictly concave as a function of $x$, so viewed as a function on $[0,1]^t$, $f(\bp)$ is a sum of concave functions and hence concave. It is thus concave also on the domain $\tP_{\rho,t}$. Thus $f(\bp)$, and hence $\tL(\rho,k,\bp)$, has a unique maximizer $\bp$. This maximizer lies in the interior of $\tP_{\rho,t}$, since the derivative of $-x\log(xd)$, namely $-\log(xd)-1$, approaches infinity as $x$ approaches~$0$.\footnote{To spell this out completely, suppose that at the maximum some $p_i=0$. To obtain a contradiction it suffices to find a direction that we can move within $\tP_{\rho,t}$ in which $p_i$ increases. Then for a small enough change in this direction, the increase in the term $-p_i\log(p_id_i)$ will outweigh the decrease in any other terms. Such a direction exists, because $\tP_{\rho,t}$ certainly contains a point $\bp'$ with $p_i'>0$, so we may choose the direction from $\bp$ to $\bp'$.}

The second statement now follows easily by the method of Lagrange multipliers, viewing $f(\bp)$ as a function on $[0,1]^t$, which we wish to maximize subject to the constraints
\begin{equation}\label{pconstr}
 \sum p_i=1\text{\quad and\quad} \sum ip_i=\rho.
\end{equation}
Indeed, we have
\[
 \frac{\partial f}{\partial p_i} = -\log(p_id_i)-1,
\]
so at the maximum there are $\lambda$ and $\mu$ such that
\[
 -\log(p_id_i)-1 = \lambda + \mu i
\]
for $1\le i\le t$. Rearranging and setting $y=-\mu$ and $x=-\lambda-1$ gives \eqref{pixy}.
The relations \eqref{xy1} and \eqref{xy2} follow immediately from the constraints \eqref{pconstr}.

Finally, to obtain \eqref{tLxy} we substitute \eqref{pixy} into the definition of $\tL_0$, noting that for this specific $\bp$ we have
\[
 \sum p_i\log(p_id_i) = \sum p_i(x+iy) = x + \rho y,
\]
again using \eqref{pconstr}.
\end{proof}

It is easy to see that, for a given integer $t$, \eqref{xy1} and \eqref{xy2} define $x$ and $y$ uniquely as functions of $\rho$ (where $1<\rho<t$), and furthermore that these functions $x(\rho)=x_t(\rho)$ and $y(\rho)=y_t(\rho)$ are (infinitely) differentiable. Indeed, dividing \eqref{xy2} by \eqref{xy1} gives
\[
 \frac{ \sum_{i=1}^t i e^{iy}d_i^{-1} } { \sum_{i=1}^t e^{iy}d_i^{-1} }  = \rho.
\]
The left-hand side is strictly increasing and (infinitely) differentiable as a function of $y$, and tends to $1$ or to $t$ as $y$ tends to $-\infty$ or $+\infty$, respectively. Having solved this equation to determine $y_t(\rho)$, we may use \eqref{xy1}, say, to find $x_t(\rho)$.

We next investigate the derivatives of $L_0$.

\begin{lemma}\label{dLrk}
For $t$ fixed the $2$-variable function $\tL_0(\rho,k,t)$ has partial derivatives
\[
 \frac{\partial}{\partial k} \tL_0(\rho,k,t) = \frac{\rho-1}{k}
\text{\quad and\quad}
 \frac{\partial}{\partial\rho} \tL_0(\rho,k,t) = \log(\rho k)-y_t(\rho).
\]
\end{lemma}
\begin{proof}
We use \eqref{tLxy}, recalling that with $t$ fixed $x=x_t(\rho)$ and $y=y_t(\rho)$ depend only on $\rho$, not on $k$. The formula for the $k$-derivative is immediate (since then $x$, $y$ and $\rho$ are constants). For the $\rho$-derivative by elementary calculus we have
\[
  \frac{\partial}{\partial\rho} \tL_0(\rho,k,t) = \log(\rho k) - \frac{\dd}{\dd\rho}(x+\rho y)
 = \log(\rho k) - \frac{\dd x_t(\rho)}{\dd\rho} - \rho\frac{\dd y_t(\rho)}{\dd\rho} -y_t(\rho).
\]
At this point something miraculous-seeming happens: if we differentiate the constraint \eqref{xy1} with respect to $\rho$ we obtain
\[
 \sum_{i=1}^t \left( \frac{\dd x_t(\rho)}{\dd\rho} + i \frac{\dd y_t(\rho)}{\dd\rho} \right ) e^{x_t(\rho)+iy_t(\rho)}d_i^{-1} =0,
\]
which, using \eqref{xy1} and \eqref{xy2}, simplifies to
\[
 \frac{\dd x_t(\rho)}{\dd\rho} + \rho \frac{\dd y_t(\rho)}{\dd\rho} =0.
\]
Combined with the formula above, this gives the result.
\end{proof}

So far, it was convenient to work in terms of $\rho$ and $k$ rather than $n$ and $k$, because certain key functions then depended only on $\rho$. However, in the end we wish to find a threshold $k^*$ as a function of $n$, so we now undo this change of variables.
Noting/recalling that the definitions \eqref{tL0def} and \eqref{L0nktdef} of $\tL_0(\rho,k,t)$ and $L_0(n,k,t)$ do not require $n$ and $k$ to be integers, for $t$ a positive integer and $n$ and $k$ positive reals with $1<n/k<t$, define
\begin{equation}\label{Lnkdef}
 \hL_0(n,k,t) := \tL_0(n/k,k,t),
\end{equation}
so, by Lemma~\ref{lscale},
\begin{equation}\label{L0hL0}
 L_0(n,k,t) = k  \tL_0(n/k,k,t) = k \hL_0(n,k,t).
\end{equation}

\begin{lemma}\label{dLnk}
For $t$ fixed the $2$-variable function $\hL_0(n,k,t)$ has partial derivatives
\[
 \frac{\partial}{\partial k} \hL_0(n,k,t) = -\frac{n}{k^2}(\log n-y_t(n/k)) + \frac{n}{k^2}-\frac{1}{k}
\text{\quad and\quad}
 \frac{\partial}{\partial n} \hL_0(n,k,t) = \frac{\log n-y_t(n/k)}{k}.
\]
\end{lemma}
\begin{proof}
This is straightforward calculus: using \eqref{Lnkdef} and the Chain Rule we have
\[
 \frac{\partial}{\partial k} \hL_0 = -\frac{n}{k^2} \frac{\partial}{\partial\rho} \tL_0 +  \frac{\partial}{\partial k} \tL_0,
\]
and
\[
 \frac{\partial}{\partial n} \hL_0 = \frac{1}{k} \frac{\partial}{\partial\rho} \tL_0.
\]
The result thus follows from Lemma~\ref{dLrk}.
\end{proof}

Our next aim is to find the value of $y$; it turns out that a fairly crude bound is enough, and for this we can use a `soft' argument, rather than trying to exactly solve the constraints \eqref{xy1} and \eqref{xy2}.
\begin{lemma}\label{lem:ybrho}
Suppose that $\rho=t-\Theta(1)$ and $\rho\ge 2$.\footnote{This condition will be irrelevant in the end; $\rho$ and $t$ will be order $\log n$. It's needed only to rule out values of $\rho$ very close to $1$.} Then
\[
 y_t(\rho) = \log\left(t2^t\right) + O(1).
\]
\end{lemma}
\begin{proof}
Note that $y_t(\rho)$ is defined for any positive integer $t$ and any real $\rho$ with $1<\rho<t$. The statement is that if we restrict the parameter space to $(\rho,t)$ such that $\rho\ge 2$ and $c<t-\rho<C$ for some constants $C>c>0$, then the difference between $y_t(\rho)$ and $\log(t2^t)$ is bounded. 

Fix, for the moment, $\rho$ and $t$ with $1<\rho<t$, and let $y=y_t(\rho)$.
Recall that $(p_i)_{i=1}^t$ with $p_i$ defined by \eqref{pixy} is a probability distribution on $[t]$ with mean $\rho$.

For $2\le i\le t$, from \eqref{pixy} we have
\begin{equation}\label{riform}
 r_i:= \frac{p_i}{p_{i-1}} = e^{y} \frac{d_{i-1}}{d_i} = \frac{e^y}{i 2^{i-1}},
\end{equation}
recalling the definition \eqref{didef} of $d_i$. In particular, $(r_i)$ is a decreasing function of $i$, so the sequence $(p_i)$ is unimodal. Furthermore, for $i=t-O(1)$ we have
\begin{equation}\label{riasymp}
 r_i =  e^y \Theta\left(\frac{1}{t 2^t}\right), 
\end{equation}
where the implicit constants do not depend on $t$ or $\rho$.
We claim that, uniformly over $(\rho,t)$ with $t-\rho=\Theta(1)$, we have $r_t=r_t(\rho,t)=\Theta(1)$; then \eqref{riasymp} gives the result.

To establish the claim suppose first (for a contradiction) that for fixed $c,C$ there exist $(\rho,t)$ with $c<t-\rho<C$ such that $r_t=r_t(\rho,t)$ is arbitrarily large. If $r_t\ge D$ then $r_i\ge D$ for all $2\le i\le t$, so (for large $D$) the sequence $p_i$ is rapidly increasing and the mean $\rho$ of this probability distribution is very close to $t$. We thus obtain a contradiction for some $D=D(c)$.

Next suppose that, with $c<t-\rho<C$, we may choose $\rho$ and $t$ such that $r_t$ is arbitrarily small. Since $r_i=\Theta(r_t)$ for $i\ge t-2C$, say, $r_i$ is also small (say $<1/2$) for $i\ge t-2C$. Thus $p_i$ decreases rapidly on $[t-2C,t]$. If $t\ge 2C+1$ then it follows that the mean of this probability distribution is less than $t-C$, a contradiction. If $t\le 2C+1$ then we conclude that $p_i$ decreases rapidly on the whole domain $[1,t]$, which implies that the mean is less than $2$, again contradicting our assumptions.
\end{proof}

 We also give a useful bound on $x+ty$, in a slightly more general form.
\begin{lemma}\label{lem:xay}
Suppose that $\rho\ge 2$, that $\rho=t-\Theta(1)$, and that $a=t+O(1)$ is a positive integer. Then
\[
 x_t(\rho)+a y_t(\rho) =\log(d_a)+O(1),
\]
where $d_a$ is defined in \eqref{didef}.
\end{lemma}
\begin{proof}
We continue the argument in the proof of the previous lemma. As shown there, defining $r_i$ as in \eqref{riform}, we have $r_i=\Theta(r_t)=\Theta(1)$ for $i=t-O(1)$. Since $(p_i)$ is a probability distribution on $[t]$ with mean $\rho$, it follows that $p_t=\Theta(1)$. Indeed, if $p_t=o(1)$ then we would have $p_i=o(1)$ for $i=t-O(1)$, contradicting that the mean $\rho$ is within $O(1)$ of $t$.

Now (purely as a notational convenience) extend the definition of $p_i$ to $i>t$ also, taking $p_i=e^{x+iy}d_i^{-1}$ as in \eqref{pixy}, with $x=x_t(\rho)$ and $y=y_t(\rho)$. Then \eqref{riform} holds for $i>t$ too, and (from this equation) we have $r_i=\Theta(r_t)$ for $i=t+O(1)$. Hence $p_a=\Theta(p_t)=\Theta(1)$. Taking logs,
\[
 \log(p_a) = x_t(\rho)+a y_t(\rho) -\log(d_a) = O(1),
\]
giving the result.
\end{proof}

We will be interested in the $\beta$-bounded chromatic number where $\beta=\alpha(n)-i=\alpha_0(n)+O(1)$, for $i=1$ (the important case for us) or $i=2$. It will turn out that the relevant values of $\rho$ (the average colour class size) are of the form $\beta-\Theta(1)$. The next corollary gives the value of $y$ in this key case.

\begin{corollary}\label{cor_y}
Suppose that $t=t(n)=\alpha_0(n)+O(1)$ is an integer. Uniformly over all $n$ and all real $\rho\ge 2$ such that $t-\rho=\Theta(1)$ we have
\begin{equation}\label{yvalue}
 y_t(\rho) = 2\log n-\log\log n + O(1).
\end{equation}
\end{corollary}
\begin{proof}
We apply Lemma~\ref{lem:ybrho}, noting that, recalling \eqref{eq:a0def},
for $t=\alpha_0(n)+O(1)$ we have $2^t=\Theta(n^2/\log^2 n)$.
\end{proof}

Using this value of $y_t(\rho)$, and Lemma~\ref{lem:xay}, we can estimate $\hL_0$ (or $\tL_0$, which is the same function reparametrized). Recall that $\hL_0$ is defined by dividing $L_0$ (a good approximation to the logarithm of the expected number of $t$-bounded $k$-colourings) by $k$, so the $+O(1)$ error below corresponds in the end to a factor $\exp(O(k))=\exp(O(n/\log n))$.

\begin{lemma}\label{L0approx}
Suppose that $k<n$ are positive reals, and $t\le a$ are positive integers, such that $a,t=\alpha_0(n)+O(1)$ and $2\le n/k=t-\Theta(1)$.
Then
\begin{equation}\label{eq:L0approx}
 \hL_0(n,k,t) = \left(a-\rho -1-\frac{2}{\log 2}\right)(\log n-\log\log n)+\log(\mu_a(n))+O(1),
\end{equation}
where $\rho=n/k$ and, as usual, $\mu_a(n)=\binom{n}{a}2^{-\binom{a}{2}}$ is the expected number of independent $a$-sets in $\Gnh$.
\end{lemma}
\begin{proof}
Let $\rho=n/k$, so by assumption $\rho\ge 2$ and $t-\rho=\Theta(1)$. By the formula \eqref{tLxy} from Lemma~\ref{lem:tLmax} we have
\[
 L:= \hL_0(n,k,t) = \tL_0(\rho,k,t) = \rho\log n -\log k -\rho +1 -(x+ay)+(a-\rho)y,
\]
where $x=x_t(\rho)$ and $y_t(\rho)$. By Lemma~\ref{lem:xay} we have
\[
 x+ay = \log(d_a)+O(1).
\]
Since
\[
 \mu_a(n) = \binom{n}{a}2^{-\binom{a}{2}} \sim \frac{n^a}{a!2^{\binom{a}{2}}} = \frac{n^a}{d_a},
\]
we have $\log(d_a)=a\log n-\log(\mu_a(n))+o(1)$. Thus
\begin{eqnarray}
 L &=& \rho\log n-\log k-\rho-a\log n+\log(\mu_a(n)) +(a-\rho)y+ O(1) \nonumber\\
 &=& (a-\rho)(y-\log n) + \log(\mu_a(n))-\log k -\rho +O(1). \label{Larho}
\end{eqnarray}
Now by assumption
\[
 \rho=n/k=t-\Theta(1)=\alpha_0(n)+O(1) = 2\log_2 n -2\log_2\log n+O(1),
\]
since $\log_2n=\Theta(\log n)$. Thus
\[
 \rho = \frac{2}{\log 2}(\log n-\log \log n) +O(1).
\]
Also, crudely, $k=n/\rho=\Theta(n/\log n)$, so
\[
 \log k = \log n-\log\log n+O(1).
\]
Substituting the last two formulae into \eqref{Larho}, we have
\[
 L = (a-\rho)(y-\log n) + \log(\mu_a(n)) -\left(1+\frac{2}{\log 2}\right)(\log n-\log\log n) +O(1).
\]
Finally, note that $a-\rho=O(1)$ and that, from \eqref{yvalue}, $y=2\log n-\log\log n+O(1)$. Thus
\[
 L = \left(a-\rho -1-\frac{2}{\log 2}\right)(\log n-\log\log n)+\log(\mu_a(n))+O(1),
\]
as claimed.
\end{proof}

We can also use the value of $y$ from Corollary~\ref{cor_y} to give approximate bounds on the partial derivatives of $L_0$ and $\hL_0=L_0/k$.

\begin{corollary}\label{dnk}
Suppose that $t=t(n)=\alpha_0(n)+O(1)$ is an integer. Uniformly over all $k\le n/2$ such that $k=n/(t-\Theta(1))$ we have
\[
 \frac{\partial}{\partial k} \hL_0(n,k,t) =  \Theta\left(\frac{\log^3 n}{n}\right)
\text{,\qquad}
 \frac{\partial}{\partial n} \hL_0(n,k,t) = -\Theta\left(\frac{\log^2 n}{n}\right)
\]
and
\[
 \frac{\partial}{\partial k} L_0(n,k,t) = \frac{2}{\log 2}\log^2 n+O(\log n\log\log n).
\]
\end{corollary}
\begin{proof}
Note that the dependence of $t$ on $n$ is only relevant for the asymptotics; by definition of partial derivative, we hold $t$ constant when differentiating. Also, in the end $t=\alpha(n)-1$ or $\alpha(n)-2$ will be locally constant.
The bounds on the partial derivatives of $\hL_0$ follow by substituting the value $y=2\log n+O(\log\log n)\sim 2\log n$ from \eqref{yvalue} into the conclusion of Lemma~\ref{dLnk}, noting that $n/k=\alpha_0(n)+O(1)=2\log_2 n+O(\log\log n)\sim 2\log_2 n$.

For $L_0(n,k,t)=k\hL_0(n,k,t)$, calculating slightly more precisely,
\begin{multline*}
 \frac{\partial}{\partial k} L_0 = 
  \frac{\partial}{\partial k} (k\hL_0) = \hL_0 + k\frac{\partial}{\partial k}\hL_0 
 = \hL_0 + \frac{n}{k}(y_t(n/k)-\log n) + \frac{n}{k}-1 \\
 = \hL_0 + \frac{n}{k}\log n +O(\log n\log\log n) 
 = \hL_0 + \frac{2}{\log 2}\log^2 n +O(\log n\log\log n).
\end{multline*}
The result follows since $\hL_0(n,k,t)=O(\log n)$ by Lemma~\ref{L0approx}.
\end{proof}

For the rest of the section we consider a function $\beta(n)$ satisfying the following assumptions; the upper bound on $\beta$ is of no particular significance.
\begin{assumption}\label{assump}
The function $\beta$ is defined on a subset $W$ of $\R$ which is a union of intervals, and is constant on each interval. Furthermore, for some constant $\eps>0$
we have
\[
 \alpha_0(n)-1-\frac{2}{\log 2} + \eps \le \beta(n) \le \alpha_0(n)+100
\]
for all large enough $n$.
\end{assumption}
Note in the assumptions of Lemma~\ref{lem:polylog}, we specified $\beta(n)=\alpha(n)-1$ or $\beta(n)=\alpha(n)-2$. These both satisfy Assumption~\ref{assump}, since $\beta\ge \alpha_0-3$ and $2/\log 2>2$.
For $n\in W$ let
\[
 I_n = \left[\frac{n}{\beta-\eps/4}, \frac{n}{\alpha_0-100}\right].
\]

Recall that $L_0(n,k,t)=k\hL_0(n,k,t)$, so one is zero if and only if the other is.
\begin{lemma}\label{lem:k*}
For each large enough (real) $n\in W$ there is a unique $k^*=k^*(n)\in I_n$ such that
\begin{equation}\label{k*cond}
 \hL_0(n,k^*(n),\beta(n)) = 0 = L_0(n,k^*(n),\beta(n)).
\end{equation}
Furthermore,
\[
 \frac{n}{k^*(n)} = \alpha(n) -1-\frac{2}{\log 2} + \frac{\log(\mu_{\alpha(n)}(n))}{\log n-\log\log n}+O(1/\log n),
\]
and if $n$ is an integer then $k^*(n)-k_\beta(n)=O(\log^2 n)$.
\end{lemma}
\begin{proof}
Keeping $n$ fixed, from Corollary~\ref{dnk}, if $n$ is large enough, then $\hL_0(n,k,\beta(n))$ is strictly increasing as a function of $k\in I_n$, with derivative $\Theta(\log^3 n/n)$. This implies uniqueness of $k^*(n)$ once we show existence.  Define $k_0=k_0(n)$ by
\[
 \frac{n}{k_0} = \alpha(n)-1-\frac{2}{\log 2} + \frac{\log(\mu_{\alpha(n)}(n))}{\log n-\log\log n}.
\]
Then, recalling that $\mu_{\alpha(n)}(n)=n^{\alpha(n)-\alpha_0(n)+o(1)}$, we have $n/k_0=\alpha_0(n)-1-2/\log 2+o(1)\le \beta(n)-\eps/2$, so $k_0\in I_n$ with $\eps/4$ room to spare.

By Lemma~\ref{L0approx} we have $\hL_0(n,k_0,\beta(n))=O(1)$; we chose $k_0$ so that the main term in \eqref{eq:L0approx} vanishes, leaving only the error term. Since, as a function of $k$, $\hL_0$ has derivative $\Theta(\log^3 n/n)$, it follows immediately that $k^*(n)$ exists, and that $k^*(n)-k_0(n)=O(n/\log^3 n)$. Since $k^*$ and $k_0$ are of order $n/\log n$, this translates to $n/k^*=n/k_0+O(1/\log n)$, proving the first statement.

For the second statement, recall the bound
\begin{equation}\label{close}
 \log(E_{n,k,\beta}) = k\hL_0(n,k,\beta) + O(\log^4 n)
\end{equation}
given by Corollary~\ref{cEnk} and \eqref{L0hL0}.
Consider $k=k^*(n)+x$, where $x$ will be of larger order than $\log^2 n$ but not too large (say $o(n/\log^2 n)$). Then from the derivative bound, $\hL_0(n,k,\beta)=\Theta(x\log^3 n/n)$,
so $k\hL_0(n,k,\beta)=\Theta(x\log^2 n)$. For $x$ of the magnitude indicated this quantity is $\omega(\log^4 n)$. Choosing such an $x$ so that $k$ is an integer, from \eqref{close} we conclude that $\log(E_{n,k,\beta})>0$, so $k_\beta(n)\le k=k^*(n)+x$. A similar argument with $x$ negative shows that $k_\beta(n)=k^*(n)+O(\log^2 n)$.
\end{proof}

\begin{lemma}\label{lem:k*diff}
The function $k^*(n)$ is differentiable on $W$, and its derivative satisfies
\[
 \left(\frac{\dd k^*(n)}{\dd n}\right)^{-1} = \frac{n}{k^*(n)} + \frac{2}{\log 2} + O(1/\log n).
\]
\end{lemma}
\begin{proof}
The Implicit Function Theorem, applied to the continuously (in fact, infinitely) differentiable function $\hL_0(n,k,t)$ with $t$ fixed tells us that $k^*(n)$, defined by $\hL_0(n,k^*,\beta(n))=0$, is differentiable, and that its derivative is $-\frac{\partial\hL_0}{\partial n} / \frac{\partial\hL_0}{\partial k}$.
Writing $\rho$ for $n/k$, by Lemma~\ref{dLnk} and Corollary~\ref{cor_y}
the reciprocal of the derivative is thus
\[
 -\frac{\partial\hL_0}{\partial k} / \frac{\partial\hL_0}{\partial n}
 = \frac{n}{k}-\frac{\rho -1}{\log n-y_\beta(\rho)}
 = \rho + \frac{\rho+O(1)}{\log n-\log\log n+O(1)}.
\]
Now $\rho=\alpha_0(n)+O(1)=2\log_2 n-2\log_2\log_2 n+O(1)$, so the last fraction above is
\[
 \frac{2\log_2 n-2\log_2\log_2 n+O(1)}{\log n-\log\log n+O(1)} = \frac{2\log_2 n-2\log_2\log_2 n+O(1)}{(\log 2)(\log_2 n-\log_2\log_2 n)+O(1)}
\]
since $\log x=(\log 2)\log_2 x$, and hence $\log\log n=\log(\log_2 n)+O(1)=(\log 2)\log_2\log_2 n+O(1)$.
The result follows.
\end{proof}

Together, Lemmas~\ref{lem:k*} and~\ref{lem:k*diff} imply Lemma~\ref{lem:polylog}, so the proof of Lemma~\ref{lem:polylog} is complete.

\subsection{Proof of Lemma~\ref{lem:as1}}\label{ss:alphashift}

We shall prove the following sharper form of Lemma~\ref{lem:as1}, since it seems that the lower bound here is perhaps quite close to the truth (see the discussion in \S\ref{s:intuition}), so this might be useful elsewhere.
\begin{lemma}\label{lem:alphashift}
Suppose that $\log^5 n\le \mu_{\alpha(n)}(n)=O(n/\log^2 n)$. Then, whp,
\[
 \chi(\Gnh) \ge k^*(n) - (1+\eps)\frac{\mu\log\nu}{\alpha(\log n-\log\log n)},
\]
where $\alpha=\alpha(n)$, $\mu=\mu(n)=\mu_{\alpha(n)}(n)$, $\nu=(n/\log n)/\mu$, $\eps=\eps(n)=O(1/\log\nu)\to 0$, and $k^*(n)=k_\beta^*(n)$ is defined in Lemma~\ref{lem:k*}.
\end{lemma}
Before giving the proof, we note that the result we need, Lemma~\ref{lem:as1}, follows.
\begin{proof}[Proof of Lemma~\ref{lem:as1}]
This is immediate from Lemma~\ref{lem:alphashift}, noting that by assumption $\nu(n)$ as defined there is $\Theta(\log n)$, so $\log\nu\sim\log\log n$, recalling that $\alpha(n)\sim c_0\log n$, and noting that by Lemma~\ref{lem:k*}, $k_\beta-k^*(n)=O(\log^2 n)$, which is much smaller than the error term we are aiming for.
\end{proof}

\begin{proof}[Proof of Lemma~\ref{lem:alphashift}]
Let 
\[
 \delta = C/\log \nu,
\]
where $C\ge 3$ is a constant that we will specify later.
Let $k_0=\floor{k^*(n)-d}$, where
\[
 d= (1+5\delta)\frac{\mu\log\nu}{\alpha(\log n-\log\log n)},
\]
so our aim is to show that whp $\chi(\Gnh)\ge k_0$. To do this, it suffices to show that whp $\Gnh$ has no proper $k_0$-colouring. Note for later that $k^*(n)=\Theta(n/\log n)$ while, recalling our assumptions on $\mu$, we have $d=O(\mu/\log n)=O(n/\log^3 n)$. Thus, crudely, $d=o(k^*/\log n)$ and it follows easily that
\begin{equation}\label{nk0close}
 n/k_0 = n/k^*(n) +o(1).
\end{equation}

Let $\alpha=\alpha(n)$. By assumption, $\mu:=\mu_\alpha(n)=O(n/\log^2 n)$, so $\mu_{\alpha+1}(n)=O(\mu\log n/n)=O(1/\log n)\to 0$, and whp $\Gnh$ contains no independent sets of size $\alpha+1$. Thus it suffices to show that whp $\Gnh$ has no $\alpha$-bounded $k_0$-colouring.

We will group the potential colourings (or, more precisely, partitions into independent sets), according to the number $m$ of $\alpha$-sets included.  Let
\[
 m^+=\mu(1+\delta).
\]
Recalling that $X_\alpha$, the number of independent $\alpha$-sets, has mean $\mu$ and variance $O(\mu)$, we know from Chebyshev's inequality that whp $X_a\le m^+$. Thus it suffices to show that whp $\Gnh$ has no $\alpha$-bounded $k_0$-colouring using at most $m^+$ $\alpha$-sets.

Let $C_m$ denote the number of partitions of $[n]=V(\Gnh)$ into exactly $k_0$ independent sets of which exactly $m$ have size $\alpha$ and none has size larger than $\alpha$. We \emph{claim} that, if $n$ is large enough, for each $m\le m^+$ we have
\begin{equation}\label{ECM}
 \E[C_m] \le 1/n.
\end{equation}
Assuming this, then summing over the $m^++1=O(\mu)=o(n)$ values of $m$ and applying Markov's inequality, the proof is complete. Thus it suffices to prove \eqref{ECM}.

Now a potential colouring/partition of the type counted by $C_m$ may be described as follows: we pick an unordered $m$-tuple of disjoint $\alpha$-vertex subsets of $[n]$, and then we pick a partition $P$ of the remaining $n-\alpha m$ vertices into $k_0-m$ parts of size at most $\alpha-1$. The partition gives a legal colouring if and only if the $m$ $\alpha$-sets are independent, and $P$ induces a legal colouring of the corresponding subgraph of $G$. Hence,
\[
 \E[C_m] = \frac{1}{m!} \binom{n}{\alpha} \binom{n-\alpha}{\alpha}\cdots\binom{n-(m-1)\alpha}{\alpha} 2^{-m\binom{\alpha}{2}} E_{n-\alpha m,k_0-m,\alpha-1},
\]
where $E_{n',k',t}$ is the expected number of $t$-bounded unordered $k'$-colourings of $G_{n',1/2}$. Hence, bounding each binomial coefficient above by $\binom{n}{\alpha}$, we have
\[
 \E[C_m] \le \frac{\mu^m}{m!} E_{n-\alpha m,k_0-m,\alpha-1}.
\]
Taking logs, and using the standard bound $\mu^m/m!\le e^\mu$ (the former is one term in the expansion of the latter), we see that
\[
 \log\E[C_m] \le \mu + \log E_m
\]
where $E_m:=E_{n-\alpha m,k_0-m,\alpha-1}$.

Fortunately, we have a good approximation for $\log E_m$. Recalling \eqref{nk0close} and noting from Lemma~\ref{lem:k*} that $n/k^*(n)\le \alpha(n)-2/\log 2<\alpha(n)-2$,
we have $n/k_0<\alpha-1$ for $n$ large enough, and it follows that
\[
 \frac{n-\alpha m}{k_0-m}\le \frac{n}{k_0}< \alpha-1.
\]
Thus we can apply Lemma~\ref{lemsimp} to conclude that
\[
 \log E_m = L_0(n-\alpha m,k_0-m,\alpha-1) + O(\log^4 n),
\]
where $L_0$ is defined in \eqref{L0nktdef}.

Unfortunately we do not have a direct formula for $L_0$ sufficiently accurate for our present purpose. Fortunately, however, we do have indirect bounds, expressed in terms of $k^*(n)$, defined in Lemma~\ref{lem:k*}. Note that we will consider a range of values $n'$ satisfying $n'\in I$, where
\[
 I = [n-\alpha m^+,n].
\]
Since $\alpha m^+=O(\alpha\mu)=O(n/\log n)$, it follows easily that $\mu_\alpha(n')=\Theta(\mu)$ for all such $n'$. In particular, $\alpha(n')=\alpha$ does not vary over this range of $n'$, and it makes sense to consider $k^*(n')$ as in Lemma~\ref{lem:k*}, defined with $\beta=\alpha-1$, as a function of $n'$.

By definition $L_0(n-\alpha m,k^*(n-\alpha m),\alpha-1)=0$ (see \eqref{k*cond}). From the last part of Corollary~\ref{dnk} we thus have
\begin{equation}\label{L0k*}
 L_0(n-\alpha m,k_0-m,\alpha-1) \sim c_0\log^2 n(k_0-m-k^*(n-\alpha m)),
\end{equation}
where $c_0=2/\log 2$. Since $k_0$ is defined in terms of $k^*(n)$, the next step is to consider how $k^*(n')$ varies as $n'$ varies between $n$ and $n-\alpha m$.

Now by Lemma~\ref{lem:polylog}, for $n'\in I$ we have
\[
 \left(\frac{\dd k^*(n')}{\dd n'}\right)^{-1} = \alpha+\frac{\log(\mu_{\alpha}(n'))}{\log n'-\log\log n'} - 1 + O(1/\log n') ,
\]
recalling that $\alpha(n')=\alpha$ for all $n'\in I$. For $n'\in I$ we have $\log n'=\log n+o(1)$ and, as noted above, $\mu_\alpha(n')=\Theta(\mu)$. It follows that
\[
 \left(\frac{\dd k^*(n')}{\dd n'}\right)^{-1} = \alpha+\frac{\log\mu}{\log n-\log\log n} - 1 + O(1/\log n)  = \alpha-\frac{\log\nu}{\log n-\log\log n} + O(1/\log n),
\]
recalling that $\nu=(n/\log n)/\mu$. Thus, 
\[
 \left(\frac{\dd k^*(n')}{\dd n'}\right)^{-1} \ge \alpha-(1+\delta)\frac{\log\nu}{\log n-\log\log n},
\]
provided the constant $C$ appearing in the definition of $\delta$ is chosen large enough.

We now take the reciprocal. Using the expansion $(\alpha-x)^{-1}=\alpha^{-1}(1-x/\alpha)^{-1}= \alpha^{-1}+x\alpha^{-2}+\cdots$ we see that for $n'\in I$ we have
\[
 \frac{\dd k^*(n')}{\dd n'} \le \frac{1}{\alpha} + (1+2\delta)\frac{\log\nu}{\alpha^2(\log n-\log\log n)}.
\]
For any $m\le m^+$ this estimate applies for all $n'$ in the interval $(n-\alpha m,n)\subset I$, so it follows immediately that
\[
 k^*(n) - k^*(n-\alpha m) \le m + (1+2\delta)\frac{m\log\nu}{\alpha(\log n-\log\log n)}.
\]
Hence
\begin{eqnarray*}
 k_0-m - k^*(n-\alpha m) &=& k_0-k^*(n)-m + (k^*(n)-k^*(n-\alpha m)) \\
 &\le& -d - m + (k^*(n)-k^*(n-\alpha m)) \\
 &\le& -d + (1+2\delta)\frac{m\log\nu}{\alpha(\log n-\log\log n)} \\
 &\le& -d + (1+4\delta)\frac{\mu\log\nu}{\alpha(\log n-\log\log n)} \\
 &\le& -\frac{\delta\mu\log\nu}{\alpha(\log n-\log\log n)} \\
 &\le& -\frac{C \mu}{\alpha\log n},
\end{eqnarray*}
where in the last three steps we used the fact that $m\le m^+=(1+\delta)\mu$, then the definition of $d$, and finally the definition of $\delta$.

Hence, from \eqref{L0k*}, if $n$ is large enough
\[
 L_0(n-\alpha m,k_0-m,\alpha-1) \le -0.99c_0\log^2 n \frac{C \mu}{\alpha\log n} \le -0.98C \mu\le -2\mu,
\]
recalling that $C\ge 3$.

Putting the pieces together, we have
\[
 \log\E[C_m] \le \mu -2\mu + O(\log^4 n) \sim -\mu,
\]
recalling that $\mu\ge \log^5 n$ by assumption. Thus, if $n$ is large enough, $\E[C_m]\le 1/n$ with plenty of room to spare, giving \eqref{ECM}. Thus the proof of Lemma~\ref{lem:as1}, and hence of Theorem~\ref{theorem:polylog}, is complete.
\end{proof}

\section{Appendix: intuition behind conjectures}\label{s:intuition}

In this section we motivate the more refined conjectures in \S\ref{ss_furtherconjs}. There are two basic starting points, both described previously, so we only recall them briefly. Firstly, the very first guess at the chromatic number is from the `expectation threshold', the least $k$ such that the expected number of partitions into $k$ independent sets is larger than $1$. In calculating this, since there are rather few profiles (a list specifying how many independent sets have each possible size) to consider, one can consider only the optimal profile.

This intuition fails immediately when we look at independent sets of size $\alpha$: the naive `optimal profile' is `unachievable', because it would like us to use $\Theta(n/\log n)$ independent sets of size $\alpha$, but the actual number $X_\alpha$ will be close to $\mu_\alpha$ which (for most $n$) will be much smaller than this. So the first approximation is to consider $\alpha$-sets separately, expecting (since the naive optimum is to use many more than there are) that we will use as many as we can, and then considering the expectation threshold for colourings without $\alpha$-sets. 

This same `unachievability' phenomenon can also arise with $(\alpha-1)$-sets; again, the optimal profile would like to use $\Theta(n/\log n)$ of them. There are certainly enough present, but not necessarily enough \emph{disjoint} ones. Numerical calculations carried out by the first author suggest that this is an issue for $\mu_{\alpha-1}(n)$ up to around $n^{1+x_0}$ for some small positive constant $x_0$.

As in \S\ref{ss_furtherconjs}, to avoid a discontinuity when $\alpha$ changes, from now on we work in terms of $a=a(n)$, chosen so that $\mu_a(n)$ is between $n^{1/2+\delta}$ and $n^{3/2-\delta}$ for some positive $\delta$. We only consider the `good' $n$, for which such an $a$ exists. Then $a=\alpha$ or $\alpha-1$. In the latter case $\mu_{\alpha}(n)$ is at most $n^{1/2-\delta}$. For us, the independent sets of size $\alpha=a+1$ can be ignored in this case: there may be enough of them to affect the chromatic number significantly, but the standard deviation of $X_\alpha$ is at most around $n^{1/4}$, which is smaller than any of our predictions for $g(n)$. Heuristically, we include all $\alpha$-sets in our colouring, but do not need to consider them any further.

As outlined above, our main heuristic (we discuss another below) is as follows: to colour we choose as large as possible a collection $\cC$ of disjoint independent sets of size $a$. Then we assume that the rest of the graph can be coloured with colour classes of size $a-1$ as predicted by the relevant expectation threshold. Let us write $m$ for $X_a$, the number of independent sets of size $a$, which will typically be $\mu_a$ plus or minus order $\sqrt{\mu_a}$, recalling that the distribution of $X_a$ is approximately Poisson, and hence asymptotically Gaussian when $\mu_a\to\infty$.
We write $t$ for the size of $\cC$. Somewhat informally, we need to understand: (I) roughly how big $t$ is, and (hence) roughly how much $t$ varies as $m$ varies, and (II) how much a given change in $t$ affects the $(a-1)$-bounded chromatic number of the remaining graph $G_{n',p}$, where $n'=n-ta$.

Let us rescale by writing
\[
 m=\frac{2xn}{a^2} \text{\quad and\quad} t=\frac{2yn}{a^2}.
\]
Rather than consider the actual distribution of independent sets of size $a$, we work heuristically in the random hypergraph model $H_a(m)$, or rather the essentially equivalent variant where the $m$ hyperedges are chosen independently and uniformly from all $a$-sets. Since two $a$-sets intersect with probability $\sim a^2/n$, we see that on average one $a$-set intersects $\sim 2x$ others.

Case 1: $x=o(1)$, i.e., $m=o(n/\log^2 n)$. Then almost all $a$-sets intersect no others, so we have $t\sim m$. Moreover, if we add an extra $a$-set, it is very likely to be disjoint from the current maximum matching, so (somewhat informally)
\[
 \frac{\dy}{\dx} = \frac{\dt}{\dm} \sim 1.
\]

Case 2: $x=\Theta(1)$. Here it is hard to say anything very precise, but it is nevertheless clear that $t=\Theta(m)$, since we still have a constant fraction of $a$-sets that intersect no others. Certainly we expect that for some\footnote{One can probably describe $h$ in terms of the size of the largest independent set in a suitable random graph $G_{n,2x/n}$, but it is not clear that this adds much. In any case, we believe we understand the asymptotic behaviour as $x\to 0$ or $x\to\infty$ from cases 1 and 3.} well-behaved increasing function $h:(0,\infty)\to(0,\infty)$ we have $y\sim h(x)$, and hence
\[
 \frac{\dy}{\dx} \sim h'(x) = \Theta(1).
\]

Case 3: $x\rightarrow \infty$. This case is more difficult, but for our heuristic we assume that the maximum matching is at least approximately given by the first moment threshold in the random hypergraph, i.e., by solving
\[
 \binom{m}{t} \frac{(n)_{at}}{(n)_a^t} \approx 1,
\]
where $(n)_k$ is the falling factorial $n!/(n-k)!$, and the ratio above is the probability that $t$ randomly chosen $a$-sets are disjoint. In turn this gives
\begin{equation}\label{eq:xy}
 \log x - \log y + 1 +o(1) = y(1+O(at/n)) = \Theta(y),
\end{equation}
and we arrive at
\[
 y = \Theta(\log x) \text{\quad and\quad} \frac{\dy}{\dx} = \Theta(1/x),
\]
with the implicit constants being $1+o(1)$ when $x=n^{o(1)}$.

Let us now turn to (II), considering how the $(a-1)$-bounded chromatic number of the rest of the graph, which we treat simply as $G_{n',p}$, $n'=n-at$, varies as $t$, and hence $n'$, varies. Heuristically, we assume the actual number of colours needed will be essentially the relevant first moment threshold, or rather the approximation $k^*$ from Lemma~\ref{lem:polylog}.

If there are $n'$ vertices left, then for each extra vertex covered by $a$-sets we expect to need $\tfrac{\dd k^*}{\dd n}|_{n=n'}$ fewer colours. We temporarily write $\gamma$ for the \emph{reciprocal} of this quantity. From Lemma~\ref{lem:polylog} and Remark~\ref{rem:noks} (which tells us that we can replace $\alpha(n)$ by $a$ in \eqref{noks}) we have
\begin{eqnarray*}
 \gamma &=& a -1 + \frac{\log\mu_a(n')}{\log n'-\log\log n'} + O(1/\log n') \\
 &=& a -1 + \frac{\log\mu_a(n')}{\log n-\log\log n} + O(1/\log n)
\end{eqnarray*}
since we'll always have $n'=\Theta(n)$.

It is convenient to work in terms of $\mu_{a+1}=\Theta(\mu_a \log n/ n)$. Let $\mu' = \mu_{a+1}(n')$, then 
\begin{eqnarray*}
 \gamma &=& a + \frac{\log\mu'}{\log n-\log\log n} + O(1/\log n) \\
 &=& a+(1+o(1)) \frac{\log\mu'}{\log n},
\end{eqnarray*} 
since we'll see later that $-\log\mu'$
is at minimum at least $\omega(1)$ (in fact at least order $\log\log n$).
This gives
\[
 \frac{\dd k^*}{\dd n} = \gamma^{-1} = a^{-1} - (1+o(1))\frac{\log \mu'}{a^2\log n}.
\]
So each extra $a$-set in the matching should save $a$ times this many colours, minus the one used for the set itself, giving `benefit' (per $a$-set used)
\[
 B \sim  \frac{-\log \mu'}{a\log n} \sim \frac{-\log \mu'}{c_0\log^2 n}
\]
where $c_0=2/\log 2$.

Now
\[
 \mu_{a+1}(n) = \Theta(\mu_a(n)\log n/n)  = \Theta(x/\log n).
\]
In all cases, writing $\approx$ for agreement up to constant factors,
\[
 \frac{\mu'}{\mu_{a+1}(n)} \sim (n'/n)^{a+1} \approx (n'/n)^a = (1-at/n)^a = (1-2y/a)^a.
\]
In cases 1 and 2, where $x$ and hence $y$ are $O(1)$, this is $\Theta(1)$ and hence irrelevant. In these cases we thus have $\mu'=\Theta(x/\log n)$, so $-\log \mu'\sim \log\log n + |\log x|$. Thus
\[
 B \sim \frac{\log\log n+|\log x|}{c_0\log^2 n}.
\]

In case 3, when $x$ grows but not too quickly, say $x=n^{o(1)}$, then $y=o(\log n)$ and hence, from \eqref{eq:xy}, $y\sim \log x$. Then 
\[
  \frac{\mu'}{\mu_{a+1}(n)} \approx (1-2y/a)^a = \exp(-(2+o(1))\log x),
\]
so $\mu'$ is roughly $1/(x\log n)$, with asymptotic agreement in the logarithms. In this case we thus obtain
\[
 B \sim \frac{\log\log n+\log x}{c_0\log^2 n}.
\]
Finally, if $x$ is at least $n^{\Omega(1)}$, then cruder estimates give $\mu'=n^{-\Omega(1)}$, so $-\log\mu'=\Theta(\log x)$. In this case
\[
 B = \Theta\left(\frac{\log x}{\log^2n}\right).
\]

In all cases, multiplying $\sqrt{\mu_a}$, our estimate for how much the number $m$ of independent $a$-sets varies, by $\tfrac{\dy}{\dx}=\tfrac{\dt}{\dm}$, and then by $B$, gives our estimate for $g(n)$, the standard deviation of $\chi(\Gnp)$.

\subsection{Complications}

In this subsection we discuss a number of issues that arise when attempting to understand the behaviour of $\chi(\Gnp)$ even more precisely. First, we should note that in any attempt at \emph{proving} Conjecture~\ref{conj:4cases}, there are major problems with the heuristic above. The key one is that, having removed some collection of independent $a$-sets, the graph that remains certainly does not have the same distribution as $G_{n',p}$ for appropriate $n'$. But even at the intuitive level, there are additional complications.

For one thing, the alert reader may have noticed that our heuristic above does not make sense in case 3 when $x$ is too large, in particular when $\mu_a\ge n^{1+x_0}$, the point up to which the naive optimum profile wants us to use more disjoint $a$-sets than can be found. Here we justify our prediction rather by the heuristic in \S\ref{ss:zigzag}. With $t$ fixed, then as $m=X_a$ varies, the number of ways of choosing $t$ (disjoint) $a$-sets varies, and this translates into variation in the chromatic number. Fortunately, for $\mu_a=n^{1+\Theta(1)}$ the two predictions agree within a constant factor, so we do not need to resolve exactly how they interact.

This same effect arises in other cases, however. Suppose we have a strategy for partially colouring with $a$-sets where we use a slightly smaller than maximum matching, so there are $N\gg 1$ choices for this matching. Then we might expect to find a colouring if the expected number of $(a-1)$-bounded colourings of the remaining $n'$ vertices is roughly $1/N$. As noted earlier, from Corollary~\ref{dnk}, for given $n'$ we should expect the extra $N$ choices to lead to a reduction in $k^*$ of around $\log N/(c_0\log^2 n')=\Theta(\log N/\log^2 n)$.

Considering the simpler case in which almost all $a$-sets are disjoint, we have $N\approx \binom{m}{t}$, so there is a large increase in the number of choices for leaving out the first few $a$-sets. Our calculations suggest that in this range we will leave out order $\Theta(\mu_a^2\log n/n)$ $a$-sets from a maximum matching. This will affect the chromatic number significantly, but we do not expect it to lead to a significant change in the variance of the chromatic number.

A further issue is that in our case $x\to\infty$, the first moment threshold is not a terribly good estimate of the size of a maximum matching of $a$-sets. In the case where $x$ does not grow too quickly, a heuristic explanation is the following. Since two $a$-sets intersect with probability $\pi_0\sim a^2/n$, we expect $t$ $a$-sets to be disjoint with probability around $\exp(-\pi_0\binom{t}{2})$. However, there is some variability in the number $M$ of overlapping pairs of $a$-sets. This quantity, which is of order $M_0=\binom{m}{2}a^2/n$, varies by around $\Delta=\sqrt{M_0}$. If we condition on this number, then our new heuristic for the probability $t$ $a$-sets are disjoint is $\exp(-\pi\binom{t}{2})$ where $\pi=M/\binom{m}{2}=\pi_0(1+\Delta/M_0)$. This variation may well be significant, and it leads to a situation where the overall expectation of the number of $t$-matchings (collections of $t$ disjoint $a$-sets) is dominated by the contribution from the case where $M$ is atypically small. Hence the first moment will not be an accurate guide to the existence of a $t$-matching. We do not explore this further here since it does not seem to affect $g(n)$. However, this, and more complicated such effects, would (at least in some cases) alter $f(n)$ by a significant amount. Thus the problem of predicting, let alone proving, a `full result' $(\chi(\Gnp)-f(n))/\sqrt{g(n)} \dto N(0,1)$ with explicit functions $f$ and $g$ seems extremely difficult.

 \bibliographystyle{plainnat}

\begin{thebibliography}{26}
\providecommand{\natexlab}[1]{#1}
\providecommand{\url}[1]{\texttt{#1}}
\expandafter\ifx\csname urlstyle\endcsname\relax
  \providecommand{\doi}[1]{doi: #1}\else
  \providecommand{\doi}{doi: \begingroup \urlstyle{rm}\Url}\fi

\bibitem[Achlioptas and Naor(2005)]{achlioptas2005two}
D.~Achlioptas and A.~Naor.
\newblock The two possible values of the chromatic number of a random graph.
\newblock \emph{Annals of Mathematics}, 162:\penalty0 1335--1351, 2005.

\bibitem[Alon and Krivelevich(1997)]{alon1997concentration}
N.~Alon and M.~Krivelevich.
\newblock The concentration of the chromatic number of random graphs.
\newblock \emph{Combinatorica}, 17\penalty0 (3):\penalty0 303--313, 1997.

\bibitem[Alon and Spencer(1992)]{alonspencerfirstedition}
N.~Alon and J.~Spencer.
\newblock \emph{The Probabilistic Method (With an Open Problems Appendix by
  Paul Erd\H os)}.
\newblock Wiley, New York, first edition, 1992.

\bibitem[Alon and Spencer(2015)]{alonspencer}
N.~Alon and J.~Spencer.
\newblock \emph{The {P}robabilistic {M}ethod}.
\newblock Wiley, 4th edition, 2015.

\bibitem[Bollob{\'a}s(1988)]{bollobas1988chromatic}
B.~Bollob{\'a}s.
\newblock The chromatic number of random graphs.
\newblock \emph{Combinatorica}, 8\penalty0 (1):\penalty0 49--55, 1988.

\bibitem[Bollob{\'a}s(2001)]{bollobas:randomgraphs}
B.~Bollob{\'a}s.
\newblock \emph{Random Graphs}.
\newblock Cambridge University Press, second edition, 2001.

\bibitem[Bollob{\'a}s(2004)]{bollobas:concentrationfixed}
B.~Bollob{\'a}s.
\newblock How sharp is the concentration of the chromatic number?
\newblock \emph{Combinatorics, Probability and Computing}, 13\penalty0
  (01):\penalty0 115--117, 2004.

\bibitem[Bollob{\'a}s and Erd{\H{o}}s(1976)]{erdoscliques}
B.~Bollob{\'a}s and P.~Erd{\H{o}}s.
\newblock Cliques in random graphs.
\newblock In \emph{Mathematical Proceedings of the Cambridge Philosophical
  Society}, volume~80, pages 419--427. Cambridge University Press, 1976.

\bibitem[Chung and Graham(1998)]{chung1998erdos}
F.~Chung and R.~Graham.
\newblock \emph{Erd{\H o}s on Graphs: his legacy of unsolved problems}.
\newblock AK Peters/CRC Press, 1998.

\bibitem[Coja-Oghlan et~al.(2008)Coja-Oghlan, Panagiotou, and
  Steger]{coja2008chromatic}
A.~Coja-Oghlan, K.~Panagiotou, and A.~Steger.
\newblock On the chromatic number of random graphs.
\newblock \emph{Journal of Combinatorial Theory, Series B}, 98\penalty0
  (5):\penalty0 980--993, 2008.

\bibitem[Erd{\H o}s and R{\'e}nyi(1960)]{erdos1960evolution}
P.~Erd{\H o}s and A.~R{\'e}nyi.
\newblock On the evolution of random graphs.
\newblock \emph{Publications of the Mathematical Institute of the Hungarian
  Academy of Sciences}, 5:\penalty0 17--61, 1960.

\bibitem[Fountoulakis et~al.(2010)Fountoulakis, Kang, and
  McDiarmid]{fountoulakis2010t}
N.~Fountoulakis, R.~Kang, and C.~McDiarmid.
\newblock The $t$-stability number of a random graph.
\newblock \emph{The Electronic Journal of Combinatorics}, 17\penalty0
  (1):\penalty0 R59, 2010.

\bibitem[Glebov et~al.(2015)Glebov, Liebenau, and
  Szab{\'o}]{glebov2015concentration}
R.~Glebov, A.~Liebenau, and T.~Szab{\'o}.
\newblock On the concentration of the domination number of the random graph.
\newblock \emph{SIAM Journal on Discrete Mathematics}, 29\penalty0
  (3):\penalty0 1186--1206, 2015.

\bibitem[Grimmett and McDiarmid(1975)]{grimmett1975colouring}
G.~R. Grimmett and C.~McDiarmid.
\newblock On colouring random graphs.
\newblock In \emph{Mathematical Proceedings of the Cambridge Philosophical
  Society}, volume~77, pages 313--324. Cambridge University Press, 1975.

\bibitem[Heckel(2018)]{heckel2018chromatic}
A.~Heckel.
\newblock The chromatic number of dense random graphs.
\newblock \emph{Random Structures \& Algorithms}, 53\penalty0 (1):\penalty0
  140--182, 2018.

\bibitem[Heckel(2021)]{heckel2019nonconcentration}
A.~Heckel.
\newblock Non-concentration of the chromatic number of a random graph.
\newblock \emph{Journal of the American Mathematical Society}, 34:\penalty0
  245--260, 2021.

\bibitem[Heckel and Panagiotou(2023)]{HPbdd}
A.~Heckel and K.~Panagiotou.
\newblock Colouring random graphs: Tame colourings.
\newblock \emph{Preprint, available at \tt{arxiv.org/abs/2306.07253}}, 2023.

\bibitem[Kang and McDiarmid(2015)]{mcdiarmidsurvey}
R.~Kang and C.~McDiarmid.
\newblock Colouring random graphs.
\newblock In \emph{Topics in Chromatic Graph Theory}, volume 156 of
  \emph{Encyclopedia of Mathematics and Its Applications}, pages 199--229.
  Cambridge University Press, 2015.

\bibitem[{\L}uczak(1991{\natexlab{a}})]{luczak1991chromatic}
T.~{\L}uczak.
\newblock The chromatic number of random graphs.
\newblock \emph{Combinatorica}, 11\penalty0 (1):\penalty0 45--54,
  1991{\natexlab{a}}.

\bibitem[{\L}uczak(1991{\natexlab{b}})]{luczak1991note}
T.~{\L}uczak.
\newblock A note on the sharp concentration of the chromatic number of random
  graphs.
\newblock \emph{Combinatorica}, 11\penalty0 (3):\penalty0 295--297,
  1991{\natexlab{b}}.

\bibitem[Matula(1970)]{matula1970complete}
D.~Matula.
\newblock On the complete subgraphs of a random graph.
\newblock In \emph{Proceedings of the 2nd Chapel Hill Conference on
  Combinatorial Mathematics and its Applications (Chapel Hill, NC, 1970)},
  pages 356--369, 1970.

\bibitem[Matula(1972)]{matula1972employee}
D.~Matula.
\newblock The employee party problem.
\newblock \emph{Notices of the American Mathematical Society}, 19\penalty0
  (2):\penalty0 A--382, 1972.

\bibitem[McDiarmid(1989)]{mcdiarmid1989method}
C.~McDiarmid.
\newblock On the method of bounded differences.
\newblock \emph{Surveys in Combinatorics}, 141\penalty0 (1):\penalty0 148--188,
  1989.

\bibitem[Panagiotou and Steger(2009)]{panagiotou2009note}
K.~Panagiotou and A.~Steger.
\newblock A note on the chromatic number of a dense random graph.
\newblock \emph{Discrete Mathematics}, 309\penalty0 (10):\penalty0 3420--3423,
  2009.

\bibitem[Scott(2008)]{scott2008concentration}
A.~Scott.
\newblock On the concentration of the chromatic number of random graphs.
\newblock \emph{Available at \tt{arxiv.org/abs/0806.0178}}, 2008.

\bibitem[Shamir and Spencer(1987)]{shamir1987sharp}
E.~Shamir and J.~Spencer.
\newblock Sharp concentration of the chromatic number on random graphs
  ${G}_{n,p}$.
\newblock \emph{Combinatorica}, 7\penalty0 (1):\penalty0 121--129, 1987.

\end{thebibliography}

\end{document}